\documentclass[11pt]{article}

  \usepackage{amsfonts,amssymb,amsthm,amsmath,cite,bbm}

\newcommand{\Z}{\mathbb Z}
\newcommand{\N}{\mathbb N}
\newcommand{\R}{\mathbb R}\newcommand{\T}{\mathbb T}
\newcommand{\V}{\mathbb V}\newcommand{\G}{\mathbb G}

\newcommand{\oZ}{\operatorname{Z}} \newcommand{\oY}{\operatorname{Y}}  
\newcommand{\oL}{\operatorname{L}} \newcommand{\oM}{\operatorname{M}} 
\newcommand{\oN}{\operatorname{N}} 
\newcommand{\oZpr}{\operatorname{Z^\prime}}

\newcommand{\oZp}{\operatorname{Z}^{\scriptscriptstyle+}}
\newcommand{\oZm}{\operatorname{Z}^{\scriptscriptstyle-}}
\newcommand{\oZpm}{\operatorname{Z}^{\scriptscriptstyle\pm}}

\def\abs#1{\vert#1\vert}

\def\cZ{\mathcal Z} 
\def\cK{\mathcal K}\def\cT{\mathcal T}

\def\rn{\R^n} 

\def\lpn{{\cal P}(\Z^n)} 
\def\lp#1{{\cal P}(\Z^{#1})} 

\def\relint{\operatorname{relint}} 
\def\relbd{\operatorname{relbd}} 

\newcommand{\slnz}{\operatorname{SL}_n(\Z)} 
\newcommand{\glnz}{\operatorname{GL}_n(\Z)} 

\def\glz#1{\operatorname{GL}_{#1}(\Z)}
\def\slz#1{\operatorname{SL}_{#1}(\Z)}

\newtheorem{theorem}{Theorem}
\newtheorem*{theo}{Theorem}
\newtheorem{corollary}[theorem]{Corollary}
\newtheorem{proposition}[theorem]{Proposition}

\newtheorem{lemma}[theorem]{Lemma}

\title{Tensor valuations on lattice polytopes}
\author{Monika Ludwig and Laura Silverstein}
\date{}

\begin{document}
\maketitle

\begin{abstract}
The Ehrhart polynomial and the reciprocity theorems by Ehrhart \& Macdonald are extended to tensor valuations on lattice polytopes.
A complete classification is established of tensor valuations of rank up to eight   that are equivariant with respect to the special linear group over the integers and  translation covariant. Every such valuation is a linear combination of the Ehrhart tensors which is shown to no longer hold true for rank nine.

\bigskip

\end{abstract}

\section{Introduction and statement of results}
Tensor valuations on convex bodies have attracted increasing attention in recent years (see, e.g., \cite{KiderlenVedelJensen,BernigHug,  HugSchneider14}). They were introduced by McMullen in \cite{McMullen:isometry} 
and Alesker subsequently obtained a complete classification of continuous and isometry equivariant tensor valuations on convex bodies (based on \cite{Alesker99} but completed in \cite{Alesker00a}).
Tensor valuations have found applications in different fields and subjects; in particular, in Stochastic Geometry and Imaging  (see \cite{KiderlenVedelJensen}). 
The aim of this article is to begin to develop the theory of tensor valuations on lattice polytopes.

Let $\lpn$ denote the set of lattice polytopes in $\R^n$; that is, the set of convex polytopes with vertices in the integer lattice $\Z^n$. In general, a full-dimensional lattice in $\rn$ is an image of $\Z^n$ by an invertible linear transformation and, therefore, all results can easily be translated to the general situation of polytopes with vertices in an arbitrary lattice.  A function $\oZ$ defined on $\lpn$ with values in an abelian semigroup
is a \emph{valuation} if 
\begin{equation*}\label{valdef}
\oZ(P)+\oZ(Q)=\oZ(P\cup Q)+\oZ (P\cap Q)
\end{equation*}
whenever $P,Q,P\cup Q,P\cap Q \in \lpn$ and $\oZ(\emptyset)=0$. 

\goodbreak
For $P\subset \R^n$, the lattice point enumerator, $\oL(P)$, is defined as
\begin{equation}\label{lpe}
  \index{lattice point enumerator}
  \oL(P)=\sum_{x\in P\cap\Z^n}1.
\end{equation} 
Hence, $\oL(P)$ is the number of lattice points in $P$ and $P\mapsto
\oL(P)$ is a valuation on $\lpn$. A function $\oZ$ defined on $\lpn$ is \emph{$\slnz\!$ invariant} if $\oZ(\phi P)= \oZ(P)$ for all $\phi\in\slnz$ and $P\in\lpn$ where $\slnz$ is the special linear group over the integers; that is, the group of transformations that can be described by $n\times n$ matrices of determinant 1 with integer coefficients. A function $\oZ$ is \emph{translation invariant} on $\lpn$ if $\oZ(P+y)= \oZ(P)$ for all $y\in\Z^n$ and $P\in\lpn$. It is \emph{$i$-homogeneous} if $\oZ(k\,P)= k^i \oZ(P)$ for all $k\in\N$ and $P\in\lpn$ where $\N$ is the set of non-negative integers.

\goodbreak
A fundamental result   on lattice polytopes by Ehrhart \cite{Ehrhart62} introduces the so-called Ehrhart polynomial and was the beginning of what is now known as Ehrhart Theory (see \cite{Barvinok2008, BeckRobins}).

\begin{theo}[Ehrhart]\label{Ehrhart}
  \index{Ehrhart polynomial}  There exist $\,\oL_i:\lpn \to \R\,$ for  $i=0,\ldots, n$ such that
  \begin{equation*}
    \oL(kP)=\sum_{i=0}^{n}\oL_i(P)k^i
  \end{equation*} 
  for every $k\in\N$ and $P\in\lpn$. For each $i$, the
  functional $\,\oL_i$ is  an $\,\slnz$ and translation invariant valuation that is homogeneous of degree $i$.
\end{theo}

\noindent
Note that $\oL_n(P)$ is the $n$-dimensional volume, $V_n(P)$, 
 and  $\oL_0(P)$  the Euler characteristic of $P$, that
is, $\oL_0(P)=1$ for $P\in\lpn$ non-empty and
$\oL_0(\emptyset)=0$. Also note that $\oL_i(P)=0$ for $P\in\lpn$ with  $\dim(P)<i$, where $\dim(P)$ is the dimension of the affine hull of $P$.

Extending the definition of the lattice point enumerator \eqref{lpe}, for $P\in\lpn$ and a non-negative integer $r$, we
define the \emph{discrete moment tensor of rank $r$} by
\begin{equation*}
  \label{discrete moment tensor}
  \oL^r(P)=\frac1{r!}\sum_{x\in P\cap\Z^n}x^r
\end{equation*}
where $x^r$ denotes the $r$-fold symmetric tensor product of $x\in\R^n$. Let
$\T^r$ denote the vector space of symmetric tensors of rank $r$ on
$\R^n$. We then have $\T^0=\R$ and $\oL^0=\oL$. For $r=1$, we obtain the {\em discrete moment vector}, which was introduced in 
\cite{BoeroeczkyLudwig}. For $r\ge2$, discrete moment tensors were introduced in \cite{BoeroeczkyLudwig_survey}.
The discrete moment tensor is a natural discretization of the \emph{moment tensor of rank $r$} of $P\in\lpn$ which is defined to be
\begin{equation}\label{moment_tensor}
    \oM^r(P)=\frac1{r!}\int_P x^r \,dx.
\end{equation}
For $r=0$ and $r=1$, respectively, this is the $n$-dimensional volume, $V_n$, and the moment vector. See \cite[Section 5.4]{Schneider:CB2} for more information on moment tensors and  \cite{HaberlParapatits_tensor, HaberlParapatits_moments, Ludwig:matrix, Ludwig:origin, Haberl:Parapatits_centro, LYZ2002, LYZ2000b} for some recent results. 
\goodbreak
 Corresponding to the theorem of Ehrhart, we establish the existence of a homogeneous decomposition for the discrete moment tensors for integers $r\ge 1$.

\begin{theorem}
  \label{tensor_ehrhart}
 There exist $\,\oL^r_i: \lpn \to \T^r$ for  $i=1,\ldots, n+r$ such
  that
  \begin{equation*}
    \oL^r(kP)=\sum_{i=1}^{n+r} \oL^r_i(P)k^i
  \end{equation*} 
  for every $k\in\N$ and $P\in\lpn$. For each $i$, the function
  $\,\oL^r_i$ is an $\slnz$ equivariant, translation covariant and $i$-homogeneous valuation.
\end{theorem}

\noindent For the definition of $\slnz$ equivariance and  translation co\-variance, see Section \ref{tools}.
The coefficients yield new valuations that we introduce here as {\em Ehrhart tensors}.
Note that  $\oL^{r}_{n+r}(P)$ is the moment tensor of $P$  and that $\oL^r_{i+r}(P)=0$ for $i>\dim(P)$ (see Section  \ref{coprop}). 
The existence of the homogeneous decomposition is proved for general  tensor valuations  in Section \ref{sec_ehrhart}.
The proof  is based on results by Khovanski{\u\i} \& Pukhlikov \cite{PukhlikovKhovanskij}.

\goodbreak
A second fundamental result on lattice polytopes is the reciprocity theorem of Ehrhart \cite{Ehrhart62} and  Macdonald \cite{Macdonald71}. 

\begin{theo}[Ehrhart \& Macdonald] \label{EMreciprocity}
 For $P\in\lpn$, the relation
$$ \oL(\relint P)=(-1)^{m}\sum_{i=0}^{m}(-1)^i \oL_i(P) $$
holds where $m=\dim(P)$.
\end{theo}

\noindent
Here, we write $\relint(P)$ for the relative interior of $P$ with respect to the affine hull of $P$.
We establish a reciprocity result corresponding to the Ehrhart-Macdonald Theorem for the discrete moment tensor.

\begin{theorem}
  \label{tensor_reciprocity}
  For $P\in\lpn$, the relation
  \begin{equation*}
    \oL^r (\relint P)=(-1)^{m+r}\sum_{i=1}^{m+r}  (-1)^i \oL^r_i(P)
  \end{equation*}
  holds where $m=\dim(P)$.
\end{theorem}
\noindent
In Section~\ref{reciprocity}, we establish reciprocity theorems for general tensor valuations; the above theorem is a special case. We follow the approach of McMullen \cite{McMullen77}.

A third  fundamental result on lattice polytopes is the Betke \& Kneser Theorem  \cite{Betke:Kneser}. It provides a complete classification of $\slnz$ and translation invariant real-valued valuations on $\lpn$ and a characterization of the Ehrhart coefficients. 

\begin{theo}[Betke \!\&\! Kneser]
A functional $\,\oZ\!:\lpn\to \R$ is  an $\,\slnz$ and translation invariant valuation if and only if  there are  $\,
c_0, \dots, c_n\in\R$ such that 
$$\oZ(P) =  c_0  \oL_0(P)+\dots+c_n \oL_n(P)$$
for every $P\in\lpn$. 
\end{theo}

\noindent
The above result was established by Betke \cite{BetkeHabil} and first published in \cite{Betke:Kneser}. In both papers, it was  assumed that the functional is invariant with respect to unimodular transformations where these are defined to be a combination of translations by integral vectors and $\glnz$ transformations; that is, linear transformations with integer coefficients and determinant $\pm1$. The proofs remain unchanged for the $\slnz$ case (see \cite{BoeroeczkyLudwig}).

\goodbreak
The Betke \& Kneser Theorem is a discrete analogue of what is presumably the most celebrated result in the geometric theory of valuations, Hadwiger's Characterization Theorem \cite{Hadwiger:V}.  Let $\cK^n$ denote the space of convex bodies (that is, compact convex sets)  on $\rn$ equipped with the topology coming from the Hausdorff metric.

\begin{theo}[Hadwiger]\label{hugo}
A functional $\,\oZ:\cK^n\to \R$ is  a continuous and rigid motion invariant valuation if and only if  there are $c_0,\ldots, c_n\in\R$ such that 
$$
\oZ(K) =  c_0 \,V_0(K)+\dots+c_n\,V_n(K)$$
for every $K\in\cK^n$. 
\end{theo}

\smallskip\noindent
Here $V_0(K),\ldots,V_n(K)$ are the {\em intrinsic volumes} of $K\in\cK^n$, which are classically defined through the Steiner polynomial. That is, for $s\ge 0$,
\begin{equation*}\label{steiner_f}
V_n(K+s \,B^n)= \sum_{j=0}^n s^{n-j} v_{n-j}\, V_j(K),
\end{equation*}
where $B^n$ is the $n$-dimensional Euclidean unit ball with volume $v_n$ and $$K+s\,B^n=\{x+s\,y:x\in K, y\in B^n\}.$$ The Hadwiger Theorem has powerful applications within Integral Geometry and Geometric Probability 
(see \cite{Hadwiger:V, Klain:Rota}). 

Hadwiger's theorem was extended to vector valuations by Hadwiger \& Schneider  \cite{Hadwiger:Schneider}.

\begin{theo}[Hadwiger \& Schneider]
A function $\,\oZ:\cK^n\to \R^n$ is  a continuous, rotation equivariant, and translation covariant valuation if and only if  there are $c_1,\ldots, c_{n+1}\in\R$ such that 
$$
\oZ(K) =  c_1 \oM^{1}_1(K)+\dots+c_{n+1}\oM^{1}_{n+1}(K)$$
for every $K\in\cK^n$. 
\end{theo}

\noindent
Here $\oM^1_i(K)=\Phi^{1,0}_i(K)$ are the \emph{intrinsic vectors} of $K$ (see (\ref{steiner_f2}) below). The key ingredient in the proof is  a characterization of the Steiner point by Schneider \cite{Schneider:steiner}.

Correspondingly, we obtain the following classification theorem for $n\ge 2$.

\begin{theorem}\label{vector}
A function $\,\oZ\!:\lpn\to \R^n$ is an $\,\slnz\!$ equivariant and translation covariant valuation  if and only if  there are $c_1,\dots, c_{n+1}\in\R$ such that $$\oZ(P) =  c_1 \oL^1_1(P)+\dots+c_{n+r} \oL^1_{n+1}(P)$$
for every $P\in\lpn$. 
\end{theorem}

\noindent
The proof is based on a characterization \cite{BoeroeczkyLudwig} of the discrete Steiner point. For $n=1$, the characterization follows from the Betke \& Kneser Theorem as only translation covariance has to be considered. Therefore, we assume $n\geq 2$ and $r\ge 1$ for the remainder of the paper.

\goodbreak
The theorems by Hadwiger and Hadwiger \& Schneider were extended by Alesker \cite{Alesker00a,Alesker98} (based on \cite{Alesker99}) to  a classification of continuous, rotation equivariant, and translation covariant tensor valuations on $\cK^n$  involving extensions of the intrinsic volumes. Just as the intrinsic volumes can be obtained from the Steiner polynomial, the moment tensor $\oM^r$ satisfies the Steiner formula 
\begin{equation}\label{steiner_f2}
\oM^r(K+s \,B^n)= \sum_{j=0}^{n+r} s^{n+r-j} v_{n+r-j} \,\sum_{k\ge 0} \Phi^{r-k,k}_{j-r+k}
\end{equation}
for $K\in\cK^n$ and $s\ge0$.
The coefficients $\Phi_k^{r,s}$ are called the {\em Minkowski tensors} (see \cite[Section~5.4]{Schneider:CB2}).
Let $Q\in\T^2$ be the metric tensor, that is, $Q(x,y)= x\cdot y$ for $x,y\in\R^n$. 

\begin{theo}[Alesker]
A function $\,\oZ:\cK^n\to \T^r$ is  a continuous, rotation equivariant, and translation covariant valuation if and only if  $\,\oZ$  can be written as linear combination of the symmetric tensor products
$Q^l \,\Phi_k^{m,s}$ with  $2l+m+s=r$.
\end{theo}

\noindent
We remark that there are linear relations, called syzygies, between the tensors described above and that the dimension of the space of continuous, rotation equivariant and translation covariant matrix valuations  $\,\oZ:\cK^n\to \T^2$ is $3n+1$ for $n\ge 2$ (see \cite{Alesker99}).

\goodbreak
For tensor valuations of rank up to eight, we obtain the following complete classification. For $r\ge3$, symbolic computation is used in the proof to show that certain matrices are non-singular.

\begin{theorem}\label{tensor}
For $2\le r\le 8$, a function $\,\oZ\!:\lpn\to \T^r$ is an $\,\slnz\!$ equivariant and translation covariant valuation  if and only if  there are $c_1,\dots, c_{n+r}\in\R$ such that $$\oZ(P) =  c_1 \oL^r_1(P)+\dots+c_{n+r} \oL^r_{n+r}(P)$$
for every $P\in\lpn$. 
\end{theorem}

\noindent While the Betke \& Kneser Theorem looks similar to the Hadwiger Theorem and Theorem~\ref{vector} looks similar to the Hadwiger \& Schneider Theorem, the similarity between the discrete  and  continuous cases breaks down for rank $r=2$, as corresponding spaces have even different dimensions.
For $n=2$ and $r=9$, there exists a new $\slz{2}$ equivariant and translation invariant valuation which is  not a linear combination of the Ehrhart tensors; it is described in Section~\ref{new}. Hence, we do not expect that a classification similar to Theorem \ref{tensor} continues to hold for $r\ge 9$. 

Additionally, we obtain a classification of translation covariant and $(n+r)$-homogeneous tensor valuations on $\lpn$ for $r\ge1$  in Theorem~\ref{hom_char} which provides a characterization of the moment tensor. The scalar case of this result corresponds to Hadwiger's classification  \cite[Satz~XIV]{Hadwiger:V} of translation invariant and $n$-homogeneous valuations on convex polytopes while the case of tensors of general rank $r$ corresponds to McMullen's  classification \cite{McMullen:isometry} of continuous, translation covariant, and  $(n+r)$-homogeneous tensor valuations on convex bodies.

\goodbreak

\section{Preliminaries}\label{tools}

For quick later reference, we aggregate most of the basics into this section and refer the reader, for more general reference, to \cite{BeckRobins, Barvinok2008, Gruber}. 

Our setting will be the $n$-dimensional Euclidean space, $\rn$, equipped with the scalar product $x\cdot y$, for $x,y\in\rn$, to identify $\rn$ with its dual space. The identification of $\rn$ with its dual space allows us to regard each symmetric $r$-tensor as a symmetric $r$-linear functional on $(\rn)^r$. Let
$\T^r$ denote the vector space of symmetric tensors of rank $r$ on
$\R^n$. We will also write this as $\T^r(\R^n)$ if we want to stress the vector space in which we are working. 

\goodbreak
The symmetric tensor product of tensors $A_i\in \T^{r_i}$ for $i=1,\dots,k$  is
\[
A_1\odot\dots\odot A_k (v_1,\dots, v_r)=\frac{1}{r!}\,\sum_\sigma A_{1}\otimes\dots\otimes A_{k}(v_{\sigma(1)}, \dots, v_{\sigma(r)})
\]
for $v_1,\dots, v_r\in\R^n$
where $r=r_1+\dots+r_k$, the ordinary tensor product is denoted by $\otimes$, and we sum over all of the permutations of $1,\dots,r$.
We use the abbreviated notation 
$AB=A\odot B$.
Specifically, the $r$-fold symmetric tensor product of $x\in\rn$ will be written as
\[
x^r=x\odot\dots\odot x.
\]
Symmetry is inherent here; so this is equal to the $r$-fold tensor product. Note that, for $x\in\rn$, its $r$-fold symmetric tensor product is
\[
x^r(v_1,\dots,v_r)=(x\cdot v_1)\cdots(x\cdot v_r)
\]
for $v_1,\dots v_r\in\rn$. We also define $x^0=1$ whenever $x\neq 0$. 

\goodbreak
Applying this to the discrete moment tensor, in particular, gives us 
\begin{equation*}
  \oL^r(P) (v_1, \dots, v_r) = \frac1{r!} \sum_{x\in P\cap\Z^n} (x\cdot v_{1}) \cdots (x\cdot v_{r})
\end{equation*}
for $v_1,\dots, v_r\in\R^n$. For the discrete moment tensor $\oL^r:\lpn \to \T^r$, the action of $\slnz$ is observed to be
\begin{equation*}
  \oL^r(\phi P) (v_1, \dots, v_r) = \oL^r(P) (\phi^t v_1, \dots, \phi^t v_r)
\end{equation*}
for $P\in\lpn$ and $\phi\in\slnz$ where $\phi^t$ is the transpose of $\phi$. In
general, a function $\oZ: $~$\lpn\to$~$ \T^r$ is
said to be {\em $\slnz$ equivariant} if
\begin{equation*}
  \oZ(\phi P) (v_1, \dots, v_r) = \oZ(P) (\phi^t v_1, \dots, \phi^t v_r)
\end{equation*}
for $v_1, \dots,v_r\in\R^n$, $\phi\in\slnz$, and $P\in\lpn$. 
We will write this as $\oZ(\phi P)= \oZ(P)\circ \phi^{t}$. We use the term $\slnz$ equivariance in order to stay consistent with the vector-valued case and note that for $x_1, \dots, x_r\in\R^n$, we have
$$\phi (x_1\odot \dots \odot x_r)= \phi\, x_1 \odot \dots \odot \phi \,x_r $$
for $\phi\in \slnz$.

Using the standard orthonormal basis of $\R^n$, which we denote as $(e_1,\dots,e_n)$, we show that an $\slnz$ equivariant tensor valuation defined on a lower dimensional lattice polytope is completely determined by its lower dimensional coordinates. 
 The precise statement is given as the following lemma. For $A\in \T^r$ and $r_j\in\N$ with $r_1+\dots+r_m=r$, we write $A(e_1[r_1], \dots, e_m[r_m])$ for $A(e_1,\dots,e_1, \dots, e_m, \dots, e_m)$ with $e_j$ appearing $r_j$ times for $j=1,\dots, m$. We identify the subspace of lattice polytopes lying in the span of $e_1, \dots, e_{m-1}$ with $\lp{m-1}$ and set $\Z^0=\{0\}$. 

\goodbreak
\begin{lemma} \label{high_coord}
If $\,\oZ:\lpn\to\T^r$ is $\slnz$ equivariant, then
\[
\oZ(P)(e_1[r_1],\dots,e_{m-1}[r_{m-1}],e_{m}[r_m])=0
\]
for every $P\in\lp{m-1}$ whenever $r_{m}>0$ and $r_1+\dots+r_m=r$.
\end{lemma}

\begin{proof}
If $m=1$, then we  consider $\phi\in\slnz$ that maps $e_{1}$ to $e_1+e_{2}$ and $e_j$ to $e_j$ for $j>1$.  For $P=\{0\}$, we have
\begin{align*}
\oZ(P)(e_1[r-1],e_2)&=\oZ(\phi P)(e_1[r-1],e_2)\\
&=\oZ(P)(e_1[r-1],e_1+e_2)\\
&=\oZ(P)(e_1[r])+\oZ(P)(e_1[r-1],e_2)
\end{align*}
yielding the result.

So, let $m\ge 2$. The proof is by induction on $r_1\ge0$.  Consider the linear transformation $\phi\in\slnz$ that maps $e_{m}$ to $e_1+e_{m}$ and maps $e_j$ to $e_j$ for all $j\neq m$. Any lattice polytope $P\in\lp{m-1}$ is invariant with respect to the map $\phi$ yielding
\begin{align*}
\oZ(P)(e_1,e_2[r_2],\dots,e_m[r_m\!-\!1])&=\oZ(\phi P)(e_1,e_2[r_2],\dots,e_m[r_m\!-\!1])\\
&=\oZ(P)(e_1+e_m,e_2[r_2],\dots,e_m[r_m\!-\!1])\\
&=\oZ(P)(e_1,e_2[r_2],\dots,e_m[r_m\!-\!1])+\oZ(P)(e_2[r_2],\dots,e_m[r_m])
\end{align*}
for any integers $r_2,\dots,r_m\geq0$ with $r_2+\dots+r_m=r$. Hence we have proved the statement for $r_1=0$. 

Let $r_1>0$ and suppose the statement holds for $r_1-1$. Then the equation 
\begin{align*}
\oZ(P)&(e_1[r_1+1], e_2[r_2],\dots,e_m[r_m])\\[3pt]
&=\oZ(\phi P)(e_1[r_1+1], e_2[r_2],\dots,e_m[r_m])\\[3pt]
&=\oZ(P)(e_1+e_m[r_1+1],e_2[r_2],\dots,e_m[r_m])\\
&=\sum_{l=0}^{r_1+1}\binom{r_1+1}{l}\oZ(P)(e_1[r_1+1-l],e_2[r_2],\dots,e_m[r_m+l])\\
&=\oZ(P)(e_1[r_1+1],e_2[r_2],\dots, e_m[r_m])+(r_1+1)  \oZ(P)(e_1[r_1],e_2[r_2],\dots, e_m[r_m+1])
\end{align*}
shows that $\oZ(P)(e_1[r_1],e_2[r_2],\dots,e_m[r_m+1])=0$, which completes the proof by induction.
\end{proof}

\goodbreak
Next, we look at the behavior of the discrete moment tensor $\oL^r$ with respect to translations.
For $y\in\Z^n$, we have
\begin{equation*}
  \oL^r(P+y)=\sum_{j=0}^r \oL^{r-j}(P) \frac{y^{j}}{j!},
\end{equation*}
where on the right side we sum over symmetric tensor products.  In accordance with McMullen
\cite{McMullen:isometry}, a valuation $\oZ: \lpn \to
\T^r$ is called \emph{translation covariant} if there exist associated
functions $\oZ^j: \lpn \to \T^j$ for $j=0,\dots, r$
such that
\begin{equation*}
  \oZ(P+y) =\sum_{j=0}^r \oZ^{r-j}(P) \frac{y^{j}}{j!}
\end{equation*}
for all $y\in\Z^n$ and $P\in\lpn$. 

\goodbreak
Certain essential properties of $\oZ$ are inherited by its associated functions. These can be seen by a comparison of the coefficients in the polynomial expansion of $\oZ$ evaluated at a translated lattice polytope. The following proposition gives the first of these properties. It was proven in \cite{McMullen:isometry} for tensor valuations on convex bodies and is included here for completeness. 

\goodbreak
\begin{proposition}\label{assoc1}
If $\,\oZ:\lpn\to\T^r$ is a translation covariant valuation with associated functions $\oZ^0, \dots, \oZ^r$, then, for $j=0, \dots, r$, the associated function $\oZ^{r-j}$ is a translation covariant valuation with the same associated functions as $\oZ$, that is,
\begin{equation*}
  \oZ^{r-j}(P+y) =\oZ^{r-j}(P)+\cdots+\oZ^0(P)\frac{y^{r-j}}{(r-j)!}
\end{equation*}
for all $y\in\Z^n$ and $P\in\lpn$.
\end{proposition}

\goodbreak
\begin{proof}
We compare coefficients in the polynomial expansion in the trans\-lation vector $y\in\Z^n$. Since $\oZ$ is a valuation, we have
\begin{align*}
\sum_{j=0}^r\oZ^{r-j}(P\cup Q)\frac{y^{j}}{j!}&=\oZ((P\cup Q)+y)=\oZ((P+y)\cup(Q+y))\\[-4pt]
&=\oZ(P+y)+\oZ(Q+y)-\oZ((P+y)\cap(Q+y))\\[6pt]
&=\oZ(P+y)+\oZ(Q+y)-\oZ((P\cap Q)+y)\\
&=\sum_{j=0}^r\oZ^{r-j}(P)\frac{y^{j}}{j!}+\sum_{j=0}^r\oZ^{r-j}(Q)\frac{y^j}{j!}
-\sum_{j=0}^r\oZ^{r-j}(P\cap Q)\frac{y^j}{j!}.
\end{align*}
Hence the associated functions of $\oZ$ are valuations.

For $y,z\in\Z^n$, observe that
\begin{align*}
\oZ(P+y+z)&=\sum_{j=0}^r\oZ^{r-j}(P+y)\frac{z^{j}}{j!}=\sum_{k=0}^r\oZ^{r-k}(P)\frac{(y+z)^k}{k!}\\
&=\sum_{k=0}^r\oZ^{r-k}(P)\sum_{j=0}^k\frac{y^{k-j}z^{j}}{j!(k-j)!}
=\sum_{j=0}^r\sum_{k=j}^r\oZ^{r-k}(P)\frac{y^{k-j}z^{j}}{j!(k-j)!}.
\end{align*}
Therefore
\[
\oZ^{r-j}(P+y)=\sum_{k=j}^r\oZ^{r-k}(P)\frac{y^{k-j}}{(k-j)!}=\oZ^{r-j}(P)+\cdots+\oZ^0(P)\frac{y^{r-j}}{(r-j)!},
\]
that is, we obtain the same associated functions as before. 
\end{proof}
 
\goodbreak
We require further results on the associated functions. 

\begin{proposition}\label{assoc2}
Let $\oZ:\lpn\to \T^r$ be a translation covariant valuation.
If $\,\oZ$ is $\slnz$ equivariant, then its associated functions are also  $\slnz$ equivariant. If $\,\oZ$ is $i$-homogeneous, then its associated function $\oZ^j$ vanishes for $j<r-i$ and otherwise is $(i+j-r)$-homogeneous.
\end{proposition}

\begin{proof}
If $\,\oZ$ is $\slnz$ equivariant, then, for any $\phi\in\slnz$, we can deduce that
\begin{align*}
\sum_{j=0}^r\oZ^{r-j}(\phi P)\frac{y^j}{j!}&=\oZ(\phi P+y)=\oZ(\phi(P+\phi^{-1}y))= \oZ(P+\phi^{-1}y) \circ \phi^t\\
&=\sum_{j=0}^r\big(\oZ^{r-j}(P) \circ \phi^t \big) \Big( \frac{(\phi^{-1}y)^j}{j!}\circ \phi^t\Big)=\sum_{j=0}^r \big( \oZ^{r-j}(P)\circ  \phi^t \big)\frac{y^j}{j!}.
\end{align*}
It follows that the associated functions are also $\slnz$ equivariant.

\goodbreak
Now suppose $\oZ$ is $i$-homogeneous and let $P\in\lpn$.  For  $k\in\N$ and $y\in\Z^n$, we have
$$
\oZ^r\!\big(k(P+y)\big)=\sum_{j=0}^r\oZ^{r-j}(kP)\frac{(ky)^{j}}{j!}.
$$
Furthermore, if we first consider the homogeneity of the valuation, we obtain
$$
\oZ^r\!\big(k(P+y)\big)=k^i\oZ^r(P+y)=\sum_{j=0}^rk^i\oZ^{r-j}(P)\frac{y^j}{j!}.
$$
As these equations hold for any $y\in\Z^n$, a comparison of the two shows that for $k\in \N$ 
$$k^j \oZ^{r-j}(kP) = k^i \oZ^{r-j}(P).$$ 
Hence,
 if the valuation  $\oZ$ is $i$-homogeneous, then  $\oZ^{r-j}$ is $(i-j)$-homogeneous for $j\le i$ and vanishes for $j>i$.
\end{proof}

\goodbreak
The inclusion-exclusion principle is a fundamental property of
valuations on lattice polytopes that was first established  by Betke but left unpublished. The first published proof
was given by McMullen in~\cite{McMullen09} where the following more general extension property was also established.  
For $m\geq 1$, we write $P_J=\bigcap_{j\in J}P_j$ for $\emptyset\neq J\subset\{1,\ldots,m\}$ and given lattice polytopes
$P_1,\ldots,P_m$. Let $\abs{J}$ denote the number of elements in $J$ and let $\G$ be an abelian
group. 

\begin{theorem}[\!\! McMullen \cite{McMullen09}]\label{McMext}
  If\, $\oZ:\lpn \to \G$ is a valuation, then
  there exists an extension of $\,\oZ$, also denoted by $\oZ$,  to finite
  unions of lattice polytopes such that 
  \begin{equation*}
    \oZ(P_1\cup \dots\cup P_m)=\sum_{\emptyset\neq
      J\subset\{1,\ldots,m\}}(-1)^{\abs{J} -1} \oZ(P_J)
  \end{equation*}
  whenever $P_J\in \lpn$ for all $\,\emptyset\neq J\subset\{1,\ldots,m\}$.
\end{theorem}

In particular, Theorem~\ref{McMext} can be used to define valuations on the relative interior of lattice polytopes. We write $\relbd P$ for the relative boundary of $P\in\lpn$ and further set
${\oZ}(\relint P)=$~$ \oZ(P)
- {\oZ}(\relbd P)$. Expressing
$\relbd P$ as the union of its faces, we obtain
\begin{equation}
  \label{intP}
  {\oZ}(\relint P)=(-1)^{\dim (P)}\sum\nolimits_{F}(-1)^{\dim(F)}\oZ(F)
\end{equation}
for $P\in\lpn$ where we sum over all non-empty faces of $P$.

Betke \& Kneser  \cite{Betke:Kneser} proved  their classification theorem  by using suitable dissections and complementations of lattice polytopes by lattice simplices. Let $T_k\in\lpn$ be the standard $k$-dimensional simplex, that is, the convex hull of the origin and the vectors $e_1, \dots, e_k$.
We call a $k$-dimensional simplex $S$  \emph{unimodular} if there are $\phi\in\slnz$ and $x\in\Z^n$ such that
$S=\phi (T_k+x)$. We require the following results.

\begin{proposition}[Betke \& Kneser \cite{Betke:Kneser}]
\label{complementation} 
For every  $P\in\lpn$, there exist unimodular simplices $S_1,\dots, S_m$ and integers $k_1,\dots,k_m$ such that
$$
\oZ(P)=\sum_{i=1}^m k_i \oZ(S_i)
$$
for all valuations $\oZ$ on $\lpn$ with values in an abelian group.
\end{proposition}

\noindent\goodbreak
The following statement is a  direct consequence of this proposition.

\begin{corollary}
\label{BetkeKneserinvariance}
If $\,\oZ, \oZpr\!: \lpn \to \T^r$ are $\,\slnz$ equivariant and translation invariant valuations such that 
\begin{equation*}\label{bkeq}
\oZ( T_i)=\oZpr (T_i) \,\text{ for }\,i=0,\dots,n,
\end{equation*}
then $\oZ=\oZpr$ on $\lpn$.
\end{corollary}

A function $\oZ$ is \textit{Minkowski additive} if 
$
\oZ(P+Q)=\oZ(P)+\oZ(Q)
$
for any $P,Q\in\lpn$. The following is proved as \cite[Remark 6.3.3]{Schneider:CB2}.

\begin{proposition}\label{linear}
Every 1-homogeneous, translation invariant valuation $\oZ:\lpn\to\T^r$ is Minkowski additive.
\end{proposition}

\section{Ehrhart tensor polynomials}\label{sec_ehrhart}

We now apply results on translative polynomial valuations to show that the evaluation of the discrete moment tensor on dilated lattice polytopes yields a homogeneous decomposition in which the coefficients themselves are new tensor valuations. In analogy to  Ehrhart's celebrated result, we call this expansion the {\em Ehrhart tensor polynomial} of~$P$. 

We now consider valuations that take values in a rational vector space which we denote by $\V$. A valuation $\oZ:\lpn\to \V$ is {\em translative poly\-nomial} of degree at most $d$ if, for every $P\in\lpn$,
the function defined on $\Z^n$ by $x\mapsto\oZ(P+x)$
is a polynomial of degree at most $d$. 
McMullen \cite{McMullen77} considered translative polynomial valuations of degree at
most one and Khovanski{\u\i} \&  Pukhlikov \cite{PukhlikovKhovanskij} proved
Theorem~\ref{Polynomial} in the general case. Another proof, following
the approach of \cite{McMullen77}, is due to Alesker \cite{Alesker99}.  These
papers assume that the valuation  on $\lpn$
satisfies the inclusion-exclusion principle, which holds by Theorem \ref{McMext}. 

\begin{theorem}[Khovanski{\u\i} \&  Pukhlikov \cite{PukhlikovKhovanskij}]\label{poly}
Let $\oZ:\lpn\to\V$ be a valuation which is translative polynomial of degree at most $d$ and let $P_1,\dots,P_m\in\lpn$ be given. For any $k_1,\dots,k_m\in\N$, the function $\oZ(k_1P_1+\dots+k_mP_m)$ is a polynomial in $k_1,\dots,k_m$ of total degree at most $d+n$. Moreover, the coefficient of $k_1^{r_1}\cdots k_m^{r_m}$ is an $r_i$-homogeneous valuation in $P_i$ which is translative polynomial of degree at most $d$.
\end{theorem}

Here we only require a special case of the result by Khovanski{\u\i} \&  Pukhlikov. 
\smallskip

\goodbreak

\begin{theorem}
 \label{Polynomial}
 If $\,\oZ:\lpn \to\V$ is a valuation that is translative polynomial of degree at most $d$, then  there exist $\oZ_i: \lpn \to \V$ for $i=0,\ldots, n+d$   such that
  \begin{equation*}
    \oZ(kP)=\sum_{i=0}^{n+d}\oZ_i(P)k^i
  \end{equation*} 
  for every $k\in\N$ and $P\in \lpn$. For each $i$, the function
  $\oZ_i$ is a translative polynomial and  $i$-homogeneous valuation.
\end{theorem}

\noindent 
Since $\oZ_i$ is $i$-homogeneous, the function $x\mapsto \oZ_i(P+x)$ is an $i$-homogeneous polynomial. As a consequence, the function $\oZ_i$ is translative polynomial of degree $i$. Note that this result contains the translation invariant case by setting $d=0$.

\goodbreak
Let $\oZ:\lpn\to \T^r$ be a  translation
covariant tensor valuation.  For given $v_1,\dots, v_r\in$~$\R^n$, we associate  the
real-valued valuation $P\mapsto \oZ(P)(v_1,\dots,v_r)$ with  the tensor valuation $\oZ$.
Since 
\begin{equation*}
  \oZ(P+y)(v_1,\dots,v_r) =\sum_{j=0}^r \Big( \oZ^{r-j}(P) \frac{y^{j}}{j!}\Big)(v_1,\dots,v_r),
\end{equation*}
the real-valued valuation $P\mapsto \oZ(P)(v_1,\dots,v_r)$ is translative polynomial of degree at most $r$. Therefore, we immediately obtain the following consequence of Theorem~\ref{Polynomial}.

\goodbreak

\begin{theorem}
  \label{tensor-poly}
  If $\,\oZ:\lpn \to\T^r$ is a translation covariant
  valuation, then  there exist $\oZ_i:$~$ \lpn \to$~$\T^r$ for $i=0,\ldots, n+r$ such that
  \begin{equation*}
    \oZ(kP)=\sum_{i=0}^{n+r}\oZ_i(P)k^i
  \end{equation*} 
  for every $k\in\N$ and $P\in \lpn$. For each $i$, the function
  $\oZ_i$ is a translation covariant  and $i$-homogeneous valuation.
\end{theorem}

\goodbreak
\noindent
Note that if the tensor valuation $\oZ$ is $\slnz$
equivariant, then so are the homogeneous components
$\oZ_0,\ldots, \oZ_{n+r}$. 

\goodbreak
The homogeneous components have a translation property that agrees with the covariance of $\oZ$. The translation covariance of $\oZ=\oZ^r$ together with its decomposition from Theorem~\ref{tensor-poly} yields
\begin{align*}
\oZ^r(k(P+y))&=\sum_{j=0}^r\oZ^{r-j}(kP)\frac{(ky)^j}{j!}\\
&=\sum_{j=0}^{r}\sum_{l=0}^{n+r-j}\oZ_{l}^{r-j}(P)\frac{k^{j+l}y^j}{j!}\\
&=\sum_{j=0}^{r}\sum_{l=j}^{n+r}\oZ_{l-j}^{r-j}(P)\frac{k^l y^j}{j!}.
\end{align*}
By the homogeneous decomposition of Theorem \ref{tensor-poly}, we also have 
\[
\oZ^r(k(P+y))=\sum_{l=0}^{n+r}\oZ_{l}^r(P+y)k^l.
\]
A comparison of the coefficients of these polynomials in $k$ gives
\begin{equation}\label{L_q^r}
\oZ_l^r(P+y)=\sum_{j=0}^l\oZ_{l-j}^{r-j}(P)\frac{y^j}{j!}
\end{equation}
where we set $\oZ_l^s=0$ for $s<0$. Furthermore, if $\oZ^r$ is $\slnz$ equivariant, then Lemma \ref{high_coord}  implies that
\begin{equation}\label{L0}
\oZ_0^r(P)=\oZ^r(0P)=\begin{cases}
c&\textup{if }r=0,\\
0&\textup{otherwise}
\end{cases}
\end{equation}
for $n\geq 2$ with $c\in\R$, and hence, $\oZ^r_1$ is translation invariant for $r\ge2$.

\goodbreak
We apply the homogeneous decomposition of Theorem \ref{tensor-poly} to the discrete moment tensor to obtain the
following corollary.

\begin{corollary}
  \label{cor:38}
   There exist $\,\oL^r_i: \lpn \to \T^r$ for $i=0,\ldots, n+r$ such
  that
  \begin{equation*}
    \oL^r(kP)=\sum_{i=0}^{n+r} \oL^r_i(P)k^i
  \end{equation*} 
 for every $k\in\N$ and $P\in\lpn$. For each $i$, the function
  $\oL^r_i$ is an $\slnz$ equivariant,  translation covariant, and $i$-homogeneous valuation.
\end{corollary}

\noindent
Theorem~\ref{tensor_ehrhart} is then an implication of Corollary \ref{cor:38} and (\ref{L0}).
We remark that, within Ehrhart Theory, further bases for the space of real-valued valuations on $\lpn$ are also important (see  \cite{BoeroeczkyLudwig_survey, JochemkoSanyal} for more information).

\goodbreak
\section{Reciprocity}\label{reciprocity}

The reciprocity theorem of Ehrhart and Macdonald~\cite{Ehrhart62, Macdonald71} is a widely used tool in combinatorics. We provide an extension of their result.

Given a function $\oZ:\lpn\to\T^r$, we define the function $\oZ^\circ:\lpn\to\T^r$ as
\begin{equation}\label{intF}
\oZ^\circ(P)=\sum\nolimits_F(-1)^{\dim (F)}\oZ(F)
\end{equation}
where the sum extends over all non-empty faces $F$ of the lattice polytope~$P$.  
Sallee~\cite{Sallee68} showed that $\oZ^\circ$ is a valuation and that $\oZ^{\circ\circ}=\oZ$. Furthermore, if $\oZ$ is translation invariant or translation covariant, then $\oZ^\circ$ has the same translation property. The latter case can be seen from the equation
\begin{align*}
\oZ^\circ(P+y)&=\sum_{F}(-1)^{\dim (F)}\oZ(F+y)=\sum_{F}(-1)^{\dim (F)}\sum_{j=0}^r\oZ^{r-j}(F)\frac{y^{j}}{j!}\\
&=\sum_{j=0}^r\sum_{F}(-1)^{\dim( F)}\oZ^{r-j}(F) \frac{y^{j}}{j!}=\sum_{j=0}^r\oZ^{(r-j)\circ}(P)\frac{y^{j}}{j!}
\end{align*}
for any $y\in\Z^n$, where we have also shown that the associated tensor $(\oZ^\circ)^{r-j}$ is equal to $(\oZ^{r-j})^\circ$ for every applicable~$j$.

\goodbreak
The following reciprocity theorem was established by McMullen \cite{McMullen77}; see \cite{JochemkoSanyal} for a different proof.  Let $\V$ be a rational vector space.

\begin{theorem}[McMullen \cite{McMullen77}]\label{McM_rec}
If $\,\oZ:\lpn\to \V$ is an $i$-homogeneous and translation invariant valuation, then
\[
\oZ^\circ(P)=(-1)^{i}\oZ(-P)
\]
for $P\in\lpn$.
\end{theorem}

This result applies to any rational vector space and, therefore, includes tensor valuations with the aforementioned properties. We now use this result to prove an analogous reciprocity theorem for translation covariant tensor valuations. 

\begin{theorem}\label{hom_rec}
If $\,\oZ:\lpn\rightarrow\T^r$ is an $i$-homogeneous and translation covariant valuation, then
\[
\oZ^\circ(P)=(-1)^{i}\oZ(-P)
\]
for $P\in\lpn$.
\end{theorem}

\begin{proof}
We prove this by induction on $r\in\N$ where the case $r=0$ is covered by Theorem~\ref{McM_rec}. By the translation behavior of $\oZ$ and $\oZ^\circ$ and by the induction hypothesis, we have for $P\in\lpn$ and $y\in\Z^n$
\begin{align*}
\oZ^\circ&(P+y)-(-1)^{i}\oZ\left(-\left(P+y\right)\right)\\[2pt]
&=\sum_{j=0}^r \oZ^{\circ j}(P)\frac{y^{r-j}}{(r-j)!}-(-1)^{i}\sum_{j=0}^r\oZ^j(-P)\frac{(-y)^{r-j}}{(r-j)!}\\
&=\oZ^{\circ r}(P)-(-1)^{i}\oZ^r(-P) 
 +\sum_{j=0}^{r-1} \Big( (-1)^{i+j-r}\oZ^{j}(-P)\frac{y^{r-j}}{(r-j)!}-(-1)^{i}\oZ^j(-P)\frac{(-y)^{r-j}}{(r-j)!}\Big)\\
&=\oZ^{\circ}(P)-(-1)^{i}\oZ(-P).
\end{align*}
Recall here that by Proposition \ref{assoc2} the associated tensor $\oZ^j$ is $(i+j-r)$-homogeneous as $\oZ$ is $i$-homogeneous and that $(\oZ^\circ)^j=(\oZ^j)^\circ$.

Let $\widetilde{\oZ}(P)=\oZ^\circ(P)-(-1)^{i}\oZ\left(-P\right)$. Then $\widetilde{\oZ}$ is an $i$-homogeneous and translation invariant valuation. From Theorem~\ref{McM_rec}, we obtain
\[
\widetilde{\oZ}^{\raisebox{-3pt}{\footnotesize$\circ$}}(P)=(-1)^i \,\widetilde{\oZ}(-P).
\]
Thus
\begin{align*}
\widetilde{\oZ}(P)&=(-1)^i\,\widetilde{\oZ}^{\raisebox{-3pt}{\footnotesize$\circ$}}(-P)=(-1)^i\left(\oZ^{\circ\circ}(-P)-(-1)^i\oZ^\circ(P)\right)\\
&=-\left(\oZ^\circ(P)-(-1)^i\oZ(-P)\right)=-\widetilde{\oZ}(P)
\end{align*}
yielding $\widetilde{\oZ}=0$ which proves the theorem.
\end{proof}

The homogeneous decomposition of  tensor valuations from Theorem~\ref{tensor-poly} allows to consider reciprocity without the assumption of homogeneity. Since $\oZ^\circ$ is also a translation covariant valuation if $\oZ$ is,
 the following result is a simple
consequence of Theorem~\ref{hom_rec}.

\begin{corollary}
If \,$\oZ:\lpn \to \T^r$ is a translation covariant valuation, then
\begin{equation*}
 \oZ^\circ( P) = \sum_{i=0}^{n+r} (-1)^{i} \oZ_i(-P)
\end{equation*}
for $P\in\lpn$.
\end{corollary}

Combined with (\ref{intP}), this gives the following result.

\begin{corollary}
If \,$\oZ:\lpn \to \T^r$ is a translation covariant  valuation, then
\begin{equation*}
 \oZ(\relint P) = (-1)^{\dim(P)} \sum_{i=0}^{n+r} (-1)^{i} \oZ_i(-P)
\end{equation*}
for $P\in\lpn$.
\end{corollary}

\noindent
So, in particular, using that $\oL^r(-P)=(-1)^r \oL^r(P)$ and that $\oL_{i+r}^r(P)=0$ for $\text{dim}(P)<i\leq n$, which is shown in Lemma \ref{coefficients} below, we obtain Theorem~\ref{tensor_reciprocity}. Note that the results in this section for translation covariant tensor valuations also hold for  translative polynomial valuations on lattice polytopes taking values in a rational vector space. The proofs remain the same.

\goodbreak

\section{Vector valuations}

For  a lattice polytope $P\in\lpn$, the discrete Steiner point $\oL_1^1(P)$ was introduced in \cite{BoeroeczkyLudwig}.
 The valuation $P\mapsto \oL_1^1(P)$  has a translation property that we refer to as translation equivariance. In general, $\oZ: \lpn\to \R^n$ is called \emph{translation equivariant} if $\oZ(P+x)= \oZ(P)+x$ for $x\in\Z^n$ and $P\in\lpn$.

\begin{theorem}[B\"or\"oczky \& Ludwig \cite{BoeroeczkyLudwig}] \label{dst}
A function $\,\oZ: \lpn\to \R^n$ is an $\,\slnz$ and translation equivariant valuation if and only if 
$\,\oZ=\oL_1^1$.
\end{theorem}

\noindent This result is the key ingredient in the classification of $\slnz$ equivariant and translation covariant vector valuations, Theorem~\ref{vector}.

\goodbreak
\subsection{Proof of Theorem~\ref{vector}}
Since $\oZ$ is translation covariant, there is $\oZ^0:\lpn \to \R$ such that
\[
\oZ(P+y)=\oZ(P)+\oZ^0(P) y
\]
for $P\in\lpn$ and $y\in\Z^n$. It follows that $\oZ^0$ is an $\slnz$ and translation invariant valuation. By the Betke \& Kneser Theorem, there are constants $c_0, \dots, c_{n}\in\R$ such that
\begin{equation*}
\oZ^0 =\sum_{i=0}^n c_{i}\, \oL_i.
\end{equation*}
Set
\[
\widetilde{\oZ} = \oZ - \sum_{i=1}^{n+1} c_{i-1} \oL^1_{i}.
\]
Note that (\ref{L_q^r}) applied to $\oL_i^1$ gives $\oL_i^1(P+y)= \oL_i^1(P)+\oL_{i-1}(P) y$. Therefore, we obtain
\begin{eqnarray*}
\widetilde{\oZ}(P+y) &=&\oZ(P+y) - \sum_{i=1}^{n+1} c_{i-1} \oL^1_{i}(P+y)\\
&=& \oZ(P) + \oZ^0(P) y -\sum_{i=1}^{n+1} c_{i-1} (\oL^1_{i}(P) +\oL_{i-1}(P)y)\\
&=& \oZ(P) +\sum_{i=0}^{n} c_{i}\oL_{i}(P)y - \sum_{i=1}^{n+1} c_{i-1}\oL^1_{i}(P)   - \sum_{i=0}^{n} c_{i}\oL_{i}(P)y\\
&=& \widetilde{\oZ}(P). 
\end{eqnarray*}
Hence $\widetilde{\oZ}$ is translation invariant and $\widetilde{\oZ}+\oL_1^1$ is $\slnz$ and  translation equivariant. Thus, Theorem \ref{dst} shows that
$\widetilde{\oZ}(P)=0$ for all $P\in \lpn$.

\goodbreak
\section{Translation invariant valuations}\label{simple}
The classification of  $\slnz$ equivariant and translation invariant tensor valuations turns out to be the main tool in our classification of $\slnz$ equivariant and translation covariant tensor valuations. In this section, we show that the only $n$-homogeneous, translation invariant tensor valuation that intertwines $\slnz$ is the trivial tensor; that is, the tensor that vanishes identically.  We also offer some definitions and include key lemmas on general translation invariant valuations here that will be applied to tensor valuations. 

\goodbreak
The following result corresponds to  Hadwiger's result \cite[Satz XIV]{Hadwiger:V} on polytopes and  is a direct consequence of a result by McMullen \cite[Theorem~1]{McMullen1978/79}. 

\begin{theorem} \label{XIV}
If $\,\oZ: \lpn\to \R$ is a trans\-lation invariant and $n$-homogeneous valuation, then there exists  $a\in\R$ such that
$$\oZ(P)=a\,V_n(P)$$ for every $P\in\lpn$.
\end{theorem}

\goodbreak
\noindent
The argument can easily be modified for tensor valuations by substituting a tensor $A\in\T^r$ for the constant $a\in\R$. Therefore, we immediately obtain the following corollary of Theorem~\ref{XIV}.
 
\begin{corollary} \label{tensor_invariant}
If $\,\oZ: \lpn\to \T^r$ is a trans\-lation invariant and $n$-homogeneous valuation, then there exists $A\in \T^r$ such that
$$\oZ(P)=A\,V_n(P) $$ for every $P\in\lpn$.
\end{corollary}

\goodbreak
As volume is $\slnz$ invariant, Corollary \ref{tensor_invariant} makes  it natural to expect that the only valuation that is translation invariant, $n$-homogeneous, and, additionally, $\slnz$ equivariant is the trivial valuation. We show that this is the case.

\begin{proposition}\label{trans_van}
Let $r\geq 1$ and $n\geq 2$. If $\,\oZ: \lpn\to \T^r$ is an $\slnz\!$ equivariant, translation invariant, and $n$-homogeneous valuation, then $\oZ(P)=0$ for every $P\in\lpn$.
\end{proposition}

\begin{proof}  By Corollary \ref{tensor_invariant}, there exists  $A\in\T^r$ such that $\oZ=V_n\, A$ on $\lpn$. For any $\phi\in\slnz$ and $v_1,\dots,v_r\in\rn$, this implies that
	$$V_n(P)\, A(v_1,\dots, v_r) =V_n(\phi P)\, A(v_1,\dots, v_r) = V_n(P) \,A(\phi^t v_1, \dots, \phi^t v_r)$$
as volume is $\slnz$ invariant and $\oZ$ is $\slnz$ equivariant. 

\goodbreak
We are left to show that the only fixed point of the action of $\slnz$ on the space of tensors is trivial. Let  $m\in\Z$ and $1\le s \le r$. If $e_1, \dots, e_n$ is a basis of $\rn$ and $j,k\in \{1,\ldots, n\}$, then, by setting $\phi^t e_j =  e_j +m\, e_k$ and $\phi^t e_l=e_l$ for $l\neq j$, we obtain a map $\phi \in \slnz$. As volume is invariant with respect to $\slnz$ transformations, for $j\neq l_{s+1}, \dots, l_r\in\{1,\dots,n\}$, we have
\begin{eqnarray*}
A(e_j [s], e_{l_{s+1}}, \dots, e_{l_r}) &=& A(e_j+ m\,e_k [s],  e_{l_{s+1}}, \dots, e_{l_r} )\\
&=& \sum_{i=0}^s \binom{s}{i} m^i A(e_j [s-i],e_k [i],  e_{l_{s+1}}, \dots, e_{l_r} ).
\end{eqnarray*}
Since $m$ is arbitrary, this implies that
$$A(e_k [s], e_{l_{s+1}}, \dots, e_{l_r})=\dots= A(e_j [s-1],e_k, e_{l_{s+1}}, \dots, e_{l_r})=0$$ 
completing the proof.
\end{proof}

\goodbreak
The Minkowski sum of $P, Q\in\lpn$ is $P+Q = \{x + y: x\in  P, y \in Q\}$. 
For $j=1, \dots, n$, a polytope $P\in\lpn$ is called a {\em $j$-cylinder} if there are proper independent linear subspaces $H_1,\dots, H_j$ of $\rn$ and lattice polytopes $P_i\subset H_i$ such that
$P = P_1 + \dots+P_j$. We denote by $\cZ_j(\Z^n)$  the class of $j$-cylinders and note  that $\cZ_n(\Z^n)\subset \dots \subset \cZ_1(\Z^n)=\lpn$. Observe that an $n$-cylinder is an $n$-dimensional parallelotope.

The following lemma can be found for convex polytopes in~\cite{KusejkoParapatits} and for lattice polytopes in \cite[Lemma 4]{McMullen77}. A valuation is called \emph{simple} if it vanishes on lower dimensional sets.

\begin{lemma}\label{zon}
If $\,\oZ:\lpn \to \R$ is a simple, translation invariant, $i$-homogeneous  valuation, then $\oZ(P)=0$ for every $P\in\cZ_j(\Z^n)$ when $j>i>0$.
\end{lemma}

\goodbreak
The following lemma can be found in~\cite{Schneider:CB2} for valuations on convex polytopes. Here, we provide a proof for lattice polytopes. Let $\V$ be a rational vector space.

\begin{lemma}\label{tinv_simple}
Let $\oZ:\lpn\rightarrow\V$ be a translation invariant valuation that is $i$-homogeneous for some $1\le i\le n$. If $P\in\lpn$ and $\dim( P)<i$, then $\oZ(P)=0$.
\end{lemma}

\begin{proof}
Let $H$ be an $(i-1)$-dimensional lattice subspace of $\R^n$. The restriction of $\oZ$ to  polytopes in $H\cap\lpn$ is a valuation on polytopes with vertices in the lattice $H\cap\Z^n$ which is invariant under the translations of $H$ into itself. The homogeneous decomposition from Theorem~\ref{Polynomial} states that this restricted $\oZ$ is a sum of valuations  homogeneous of degrees $0,\dots,i-1$. However, the valuation $\oZ$ is $i$-homogeneous implying that $\oZ(P)=0$ for $P\subset H$. The translation invariance of $\oZ$ together with the arbitrary choice of $H$ implies that $\oZ(P)=0$ for every  $P\in\lpn$ such that $\dim P<i$.
\end{proof}

\goodbreak
\section{Properties of the Ehrhart tensors}\label{coprop}

The Ehrhart tensors $\oL_i^r$ include, for $r=0$, and expand upon the Ehrhart coefficients. They are also the discrete analogues of the Minkowski tensors. By Propositions \ref{assoc1} and \ref{assoc2}, the Ehrhart tensors $\oL_i^r:\lpn\to \T^r$ are $\slnz$ equivariant and translation covariant valuations.  In this section, we derive further properties of these tensors and give a characterization for the leading Ehrhart tensor. 

Recall that the discrete moment tensor is the discrete analogue of the moment tensor defined in (\ref{moment_tensor}). On lattice polytopes, the moment tensor coincides with the leading Ehrhart tensor. The following result is well-known for $r=0$ (where $\oM^0=V_n$).

\begin{lemma}\label{coefficients}
For $P\in\lpn$, 
\[
\oL_{n+r}^{r}(P)=\oM^r(P)
\]
and $\oL^r_{i+r}(P)=0$ for $\dim(P)<i\leq n$. Moreover, $\oL^r_{i+r}$ is not simple for $0\le i<n$.
\end{lemma}

\begin{proof}
By the definition of the Riemann integral, we have 
\begin{align*}
\oL^{r}_{n+r}(P)
&= \lim_{k\to\infty}\frac{\oL^{r}(kP)}{k^{n+r}}
= \frac1{r!}\, \lim_{k\to\infty}\,\,\frac{1}{k^n}\sum_{x\in kP\cap\Z^n}\frac{1}{k^r}x^{r}\\
&=\frac1{r!}\, \lim_{k\to\infty} \,\,\frac{1}{k^n}\!\sum_{x\in P\cap\tfrac{1}{k}\Z^n}x^{r}
=\frac1{r!}\, \int\limits_Px^{r}\,dx.
\end{align*}
This proves the statement for $\dim(P)=n$ and shows that $\oL_{n+r}(P)=0$ for $\dim(P)<n$. The statement for $\dim(P)<i$ follows by considering the affine hull of $P$ since $\oL^r_{i+r}$ is again proportional (with a positive factor) to the moment tensor calculated in this subspace. This also implies that $\oL^r_{i+r}$ is not simple for $0\le i<n$.
\end{proof}

Although our main interest in this article is the classification of $\slnz$ equivariant and translation covariant tensor valuations, we also obtain a characterization of translation covariant and $(n+r)$-homogeneous tensor valuations. In fact, by Lemma~\ref{coefficients}, they are equal to the moment tensor up to a scalar. In this simple result, which is analogous to Alesker's result on tensor valuations on convex bodies in~\cite{Alesker00a}, no $\slnz$ equivariance is assumed. 

\goodbreak
\begin{theorem}\label{hom_char}
If $\oZ:\lpn\rightarrow\T^r$ is a translation covariant, $(n+r)$-homogeneous valuation, then there is $c\in\R$ such that
\[
\oZ(P)=c\oL_{n+r}^r(P)
\]
for every $P\in\lpn$.
\end{theorem}

\goodbreak
\begin{proof}
For $r=0$, the statement is the same as Theorem~\ref{XIV}.  Suppose the assumption is true for all translation covariant and $(n+i)$-homogeneous valuations that take values in tensors of rank $i<r$. Let $\oZ^{r-j}$ for $j=1,\dots, r$ be the associated functions of $\oZ$. By Proposition \ref{assoc2}, the associated function $\oZ^{r-1}$ is homogeneous of degree $n+r-1$. Hence, by the induction assumption, there is $c\in\R$ such that $\oZ^{r-1}=c\oL_{n+r-1}^{r-1}$. Note that it follows from  Proposition~\ref{assoc1} that 
\begin{equation}\label{cconst}
\oZ^{r-j}=c\,\oL^{r-j}_{n+r-j}
\end{equation}
for $j=1, \dots, r-1$. Consider the translation covariant and $(n+r)$-homogeneous valuation $\widetilde \oZ=\oZ-c\oL_{n+r}^r$. For $y\in\Z^n$, by the translation covariance of $\widetilde{\oZ}$ and $\oL_{n+r}^r$  and (\ref{cconst}), we obtain
\begin{align*}
\widetilde \oZ(P+y)&=\oZ(P+y)-c\oL_{n+r}^r(P+y)\\
&=\sum_{j=0}^r\oZ^{r-j}(P)\frac{y^j}{j!}-c\sum_{j=0}^r\oL_{n+r-j}^{r-j}(P)\frac{y^j}{j!}\\
&=\oZ(P)-c\oL_{n+r}^r(P)\\[2pt]
&=\widetilde \oZ(P).
\end{align*}
Therefore, the valuation $\widetilde\oZ$ is translation invariant. Theorem~\ref{Polynomial} implies that $\widetilde{\oZ}=0$ as non-trivial translation invariant valuations cannot be homogeneous of degree greater than $n$.
\end{proof}

The characterization of the first Ehrhart tensor is the key element in the classification of tensor valuations. 
We show, in Lemma~\ref{nontriv}, that it can only be simple in the planar case.
Faulhaber's formula oftentimes appears in the calculation of the discrete moment tensor of a lattice polytope as it does in Lemma~\ref{nontriv}. 
The formula was given by Bernoulli in {\em Ars Conjectandi} which was translated in~\cite{Bernoulli} although he fully attributed it to Faulhaber due to his formulas for sums of integral powers up to the 17th power~\cite{Faulhaber}. With the convention that $B_1=-\frac{1}{2}$ and that $B_{2i+1}=0$ for $i>0$, the formula is stated as
\begin{equation}\label{Faulhaber}
\sum_{i=1}^k i^r=\frac{1}{r+1}\sum_{l=0}^r(-1)^l\binom{r+1}{l}B_lk^{r+1-l}
\end{equation}
where $B_l$ are the Bernoulli numbers. We will use the following convolution identity for Bernoulli polynomials (see, e.g.,~\cite{AgohDilcher}) which, interestingly enough, is usually attributed to Leonhard Euler. For $n\geq 1$, the identity is
\begin{equation}\label{Euler}
\sum_{i=0}^n\binom{n}{i}B_iB_{n-i}=-nB_{n-1}-(n-1)B_n.
\end{equation}

\goodbreak
\begin{lemma}\label{nontriv}
For $n\geq 2$, the valuation $\oL_1^r$ is non-trivial. For $n=2$ and $r\geq 3$ odd, it is simple.
\end{lemma}

\begin{proof}
It suffices to prove that $\oL_1^r(T_2)(e_1[r])\neq0$ and, by Lemma~\ref{high_coord} for $n=2$ and $r\geq 3$ odd, that $\oL_1^r(T_1)(e_1[r])=~0$. 

For any $k\in\N$, we have
\begin{equation}\label{LT1}
\oL^r(kT_1)(e_1[r])=\sum_{x\in kT_1\cap\Z^n}(x\cdot e_1)\cdots(x\cdot e_1)=\sum_{i=1}^k i^r
\end{equation}
where the sum of the first $k$ powers of $r$ can be expressed through Faulhaber's formula~(\ref{Faulhaber}). By its homogeneous decomposition, Corollary~\ref{cor:38}, the first Ehrhart tensor is the coefficient of $k$, where $l=r$ in~(\ref{Faulhaber}), implying that $\oL_1^r(T_1)(e_1[r])=(-1)^rB_r$. As $B_r=0$ for $r=2m+1$ where $m\neq 0\in\N$, we obtain the second statement of the lemma.

Similarly to~(\ref{LT1}), for any $k\in\N$, we have
\[
\oL^r(kT_2)(e_1[r])=\sum_{x\in kT_2\cap\Z^n}(x\cdot e_1)\cdots(x\cdot e_1)=\sum_{j=1}^k\sum_{i=1}^j i^r.
\]
Applying Faulhaber's formula~(\ref{Faulhaber}) twice, the discrete moment tensor of $kT_2$ is
\begin{align}\label{faul}
\oL^r(kT_2)(e_1[r])&=\sum_{j=1}^k\sum_{i=1}^j i^r \nonumber\\
&=\frac{1}{r+1}\sum_{l=0}^r(-1)^l\binom{r+1}{l}B_l\sum_{j=1}^k j^{r+1-l}\\
&=\frac{1}{r+1}\sum_{l=0}^r\frac{(-1)^l B_l}{r+2-l}\binom{r+1}{l}\sum_{m=0}^{r+1-l}(-1)^m\binom{r+2-l}{m}B_mk^{r+2-l-m}. \nonumber
\end{align}
The value of $\oL_1^r(T_2)(e_1[r])$ is equal to the coefficient of $k$ in~(\ref{faul}); precisely the value when we set $m=r+1-l$. Hence
\begin{align}\label{faul2}
\oL_1^r(T_2)(e_1[r])&=\frac{(-1)^{r+1}}{r+1}\sum_{l=0}^{r}\binom{r+1}{l}B_lB_{r+1-l}.
\end{align}
Euler's identity~(\ref{Euler}) together with equation~(\ref{faul2}) and $B_0=1$ then gives
\[
\oL_1^r(T_2)(e_1[r])=(-1)^{r}\left(B_r+B_{r+1}\right)\neq 0
\]
as $-(B_1+B_2)=\frac{1}{3}$ and, for any $m\neq0\in\N$, $B_{2m}\neq 0$ and $B_{2m+1}=0$.
\end{proof}

\section{The classification of tensor valuations}

The main aim of this section is to prove Theorem~\ref{tensor}. We also obtain characterization results for the one-homogeneous component of the discrete moment tensor in Corollaries \ref{trans_inv} and \ref{mink_add} and construct a new $\slz{2}$ equivariant and translation invariant valuation $\oN: \lp{2}\to\T^9$ in Section \ref{new}.  We start with a discussion of simple tensor valuations in the planar case.

\goodbreak
\subsection{Simple tensor valuations on $\lp{2}$}\label{simple_plane_section}

We make the following elementary observation for valuations that vanish on the square $[0,1]^2$.

\begin{lemma}\label{even_tensor_simple_planar}
Let $r>1$ be even and let $\oZ:\lp{2}\rightarrow\T^r$ be a simple, $\slz{2}$ equivariant, and translation invariant valuation. If $\,\oZ([0,1]^2)=0$, then $\oZ(T_2)=0$.
\end{lemma}

\begin{proof}
The square $[0,1]^2$ can be dissected into $T_2$ and a translate of $-T_2$. Therefore, we obtain
\[
\oZ(T_2)+\oZ(-T_2)=(1+(-1)^r)\oZ(T_2)=0
\]
which implies that $\oZ(T_2)=0$.
\end{proof}
\goodbreak

We also require the following result.

\begin{lemma}\label{odd_tensor_simple_planar}
Let $1<r<8$ be odd and let $\oZ:\lp{2}\rightarrow\T^r$ be a simple, $\slz{2}$ equivariant, and translation invariant valuation. If $\,\oZ([0,1]^2)=0$, then there exists $c\in\R$ such that $\oZ(T_2)=c\oL_1^r(T_2)$.
\end{lemma}

\begin{proof}
Let $\varepsilon\in\{0,1\}$.  Following \cite{HaberlParapatits_tensor}, a valuation $\oZ$ is called $\glz{2}$-$\varepsilon$-equivariant, if 
$$\oZ(\phi P)= (\det \phi)^\varepsilon \oZ(P)\circ \phi^t$$ 
for all $\phi\in\glz{2}$ and $P\in\lp{2}$, where $\det$ stands for determinant. 

\goodbreak
Let $\vartheta\in \glz{2}$ be the transform that swaps $e_1$ with $e_2$ and hence has $\det \vartheta =-1$. Defining, as in \cite{HaberlParapatits_tensor}, the valuations $\oZp$ and $\oZm$  for $P\in\lp{2}$ by
\begin{eqnarray*}
\oZp(P)&=& \tfrac12 \big( \oZ(P) + \oZ(\vartheta^{-1} P)\circ \vartheta^t\big),\\
\oZm(P)&=& \tfrac12 \big( \oZ(P) - \oZ(\vartheta^{-1} P)\circ \vartheta^t\big),
\end{eqnarray*}
we see that 
$\oZp$ is $\glz{2}$-$\varepsilon$-equivariant with $\varepsilon=0$ and that $\oZm$ is $\glz{2}$-$\varepsilon$-equivariant with $\varepsilon=1$. Indeed, if $\phi\in\slz{2}$ and $P\in\lp{2}$, then
\begin{eqnarray*}
\oZp(\phi P)&=& \tfrac12 \big( \oZ(\phi P) + \oZ((\vartheta^{-1}\phi \vartheta) \vartheta^{-1}  P)\circ \vartheta^t\big)\\
&=& \tfrac12 \big( \oZ( P) \circ \phi^t + \oZ(\vartheta^{-1}  P)\circ \vartheta^t \phi^t\big)\\
&=&  \oZp(P)\circ \phi^t.
\end{eqnarray*}
If $\phi\in\glz{2}$ with $\det\phi =-1$ and $P\in\lp{2}$, then 
\begin{eqnarray*}
\oZp(\phi P)&=& \tfrac12 \big( \oZ(\phi\vartheta \vartheta^{-1} P) + \oZ(\vartheta^{-1}  \phi P)\circ \vartheta^t\big)\\
&=& \tfrac12 \big( \oZ(\vartheta^{-1} P) \circ\vartheta^t \phi^t + \oZ(P)\circ \phi^t\big)\\
&=& \oZp(P)\circ \phi^t.
\end{eqnarray*}
The proof for $\oZm$ is similar. Moreover, note that $\oZ= \oZp+\oZm$.

\goodbreak
Let $r=2s+1$ for $s\in\N$. We set  $a_{r_1} =\oZp(T_2)(e_1[r_1],e_2[r_2])$ for $0\le r_1\le r$ and $r_1+r_2=r$. Then
\[
\oZp(T_2)(e_1[r_1],e_2[r_2])=\oZp(\vartheta T_2)(e_1[r_1],e_2[r_2])=\oZp(T_2)(e_1[r_2],e_2[r_1])
\]
or $a_{r_1}=a_{r-r_1}$. If we set  $b_{r_1} =\oZm(T_2)(e_1[r_1],e_2[r_2])$ for $0\le r_1\le r$ and $r_1+r_2=r$, then
\[
\oZm(T_2)(e_1[r_1],e_2[r_2])=\oZm(\vartheta T_2)(e_1[r_1],e_2[r_2])=-\oZm(T_2)(e_1[r_2],e_2[r_1])
\]
or $b_{r_1}=-b_{r-r_1}$. 
Thus, in each case,  we have to determine only $s+1$ coordinates of $\oZpm(T_2)$.

Let $\phi\in\slnz$ be the map sending $e_1$ to $-e_2$ and $e_2$ to $e_1-e_2$.  We have $T_2-e_2=\phi T_2$. For $0\leq r_1\leq r$, the translation invariance of $\oZpm$ implies
\begin{align*}
\oZpm(T_2)(e_1[r_1],e_2[r_2])&=\oZpm(\phi T_2)(e_1[r_1],e_2[r_2])\\[4pt]
&=\oZpm(T_2)(e_2[r_1],-e_1-e_2[r_2])\\
&=(-1)^{r_2} \sum_{i=0}^{r_2} \binom{r_2}{i}\oZpm(T_2)(e_2[r-i],e_1[i]).
\end{align*}

First, we look at $\oZm(T_2)$. Note that $r_1=0$ gives us $b_0=-b_r$ and that we have a system of $r+1$ equations involving $b_0,\dots,b_{r}$. That is, for $r_1$ odd, we have
\begin{equation}\label{Zp1}
b_{0}+\binom{r_1}{1}b_1+\dots+\binom{r_1}{r_1-1}b_{r_1-1}+2b_{r_1}=0
\end{equation}
and, for $r_1>0$ even, 
\begin{equation}\label{Zp2}
b_{0}+\binom{r_1}{1}b_1+\dots+\binom{r_1}{r_1-1}b_{r_1-1}=0.
\end{equation}
It is easily checked that, for $1<r<8$ odd, this system of equations combined with $b_{r_1}=-b_{r-r_1}$ has rank $r+1$. Hence $\oZm(T_2)$ vanishes and we have $\oZ(T_2)=\oZp(T_2)$. Yet, equations (\ref{Zp1}) and (\ref{Zp2}) remain the same for $\oZp(T_2)$ with the replacement of each $b_i$ by $a_i$.  It is easy to see that for $1<r<8$ odd, this system of equations combined with $a_{r_1}=a_{r-r_1}$  has rank $r$. As the tensor $\oL_1^r(T_2)$ is non-trivial by Lemma \ref{nontriv},  any solution is a multiple of $\oL_1^r(T_2)$ concluding the proof.
\end{proof}

We remark that the above lemma fails to hold for $r>8$ odd. In particular, the system of equations that determine the $(r+1)$ coordinates $\oZ(T_2)(e_1[r_1],e_2[r_2])$ with $r_1+r_2=r$ has rank $(r-1)$ for $r=9, 11, 13$ and rank $(r-2)$ for $r=15, 17 ,19$; there exist new tensor valuations in these cases.  For $r=9$, we describe the construction of this new valuation  in the following section.

\begin{proposition}\label{tensor_simple_planar}
Let $2\le r\le 8$.
If $\,\oZ:\lp{2}\rightarrow\T^r$ is a simple, $\slz{2}$ equivariant, and translation invariant valuation, then $\oZ=0$ for $r$ even and  there is $c\in\R$ such that $\oZ=c \oL_1^r$ for $r$ odd.
\end{proposition}

\begin{proof}
We only need to consider the statement for $\oZ$ being in addition $i$-homogeneous by Theorem~\ref{Polynomial}. 
If $i=2$, then $\oZ$ is trivial due to Proposition~\ref{trans_van}. If $i=1$, then Lemma \ref{zon} implies that $\oZ$ vanishes on $\cZ_2(\Z^n)$ which gives $\oZ([0,1]^2)=0$. By Lemma \ref{even_tensor_simple_planar}, we have $\oZ(T_2)=0$ for $r$ even. By Lemma \ref{odd_tensor_simple_planar}, there is $c\in\R$ such that $\oZ(T_2)=c\oL_1^r$ for $r$ odd. 
Since $\oZ$ is simple, Corollary \ref{BetkeKneserinvariance} implies in both cases the result.
\end{proof}

\goodbreak
\subsection{A new tensor valuation on $\lp{2}$}\label{new}

We now define a new simple, 1-homogeneous, $\slz{2}$ equivariant, and translation invariant tensor valuation $\oN: \lp{2} \to \T^9$. The basic step is to set $\oN(T_2)= \oL_1^3(T_2)^3$; that is, to use the threefold symmetric tensor product of $\oL^3_1(T_2)$. Note that $\oL_1^3: \lp{2}\to \T^3$ is simple, $\slz{2}$ equivariant, translation invariant, and, by Lemma \ref{nontriv}, non-trivial. So, for $\phi\in\slz{2}$, we have
\begin{eqnarray}\label{equivariance}
\oL_1^3(\phi T_2)^3&=&\oL_1^3(\phi T_2)\odot \oL_1^3(\phi T_2)\odot\oL_1^3(\phi T_2)\nonumber\\
&=&\oL_1^3(T_2) \circ \phi^t \odot \oL_1^3(T_2)\circ \phi^t\odot \oL_1^3(T_2)\circ \phi^t\\
&=&\oL_1^3(T_2)^3\circ \phi^t.\nonumber
\end{eqnarray}
Also note that $\oL_1^3([0,1]^2)=0$ by Lemma \ref{zon}.

More precisely, we set $\oN(P)=0$ for $P\in\lp{2}$ with $\dim(P)\le 1$ and for a two-dimensional lattice polygon $P$ we choose a dissection into translates of triangles $\phi_i T_2$ with $\phi_i\in\slz{2}$ for $i=1, \dots, m$. 
Here $P$ is said to be {\em dissected} into the triangles $S_1, \dots, S_m$ if $P=S_1\cup \dots \cup S_m$ where $S_i$ and $S_j$ have disjoint interiors for every $i\ne j$; this is written as $P= S_1\sqcup\dots\sqcup S_m$.
By Theorem~\ref{McMext}, a simple valuation $\oN$ then has the property that $\oN(S_1\sqcup\dots\sqcup S_m)=\oN(S_1)+\dots+\oN(S_m)$.
We set
\begin{equation}\label{newval}
\oN(P) = \sum_{i=1}^m \oL_1^3(\phi_i T_2)^3.
\end{equation}
Note that (\ref{equivariance}) implies that $\oN$ is $\slz{2}$ equivariant. 

We need to show that $\oN$ is well-defined. To do this, we use the following definition and  theorem.
Every lattice polygon has a unimodular triangulation; that is, a dissection into unimodular triangles (see, e.g.,~\cite{DeLoeraRambauSantos}). If the union of two unimodular triangles in such a triangulation is a convex quadrilateral $Q$, then replacing the diagonal of $Q$ given by the edges of the adjacent triangles with the opposite diagonal produces a new unimodular triangulation. This process is called a {\em flip}. 

\begin{theorem}[\!\! Lawson \cite{Lawson}]\label{flipflop}
Given any two unimodular triangulations $\cT$ and $\cT'$ of a lattice polygon $P\in\lp2$, there exists a finite sequence of flips transforming $\cT$ into $\cT'$.
\end{theorem}

\goodbreak
We now show that the definition (\ref{newval}) does not depend on the choice of the triangulation. 

\begin{lemma}\label{flip}
Let $S_1, \dots, S_m$ and $S_1',\dots, S_m'$ be unimodular triangles.
If $$S_1\sqcup \dots \sqcup  S_m=S_1' \sqcup \dots\sqcup S_m'\in \lpn,$$ 
then
$$\sum_{i=1}^m\oN(S_i)=\sum_{i=1}^m  \oN(S_i').$$
\end{lemma}

\begin{proof}  By Theorem \ref{flipflop}, there is a sequence of flips that transforms any triangulation of a given polygon $P$ to any other triangulation of P. Therefore, it suffices to check that the value of $\oN$ is not changed by any flip, as $\oN$ vanishes on lower dimensional polygons. So if $S_i\sqcup S_j= S_k' \sqcup S_l'$  and $S_i\sqcup S_j$ is a translate of an $\slz2$  image of $[0,1]^2$, we have to show that
$$\oN(S_i)+ \oN(S_j) = \oN(S_k')+ \oN(S_l').$$
This is easily seen. Indeed, since $S_i\sqcup S_j$ is a translate of an $\slz2$ image of  $[0,1]^2$, we have 
$$\oN(S_i)+\oN(S_j)=\oN(S_i\sqcup S_j)=0,$$
by the $\slz{2}$ equivariance as $\oL_1^3([0,1]^2)=0$ and the same holds for $S_k', S_l'$.
\end{proof}

\goodbreak
Lemma \ref{flip} also shows that $\oN$ is a valuation. Indeed, if $P,Q \in \lp{2}$ are such that $P\cup Q\in\lp{2}$ and $\cT$ is a triangulation of $P\cup Q$, then we perform a sequence of flips on $\cT$ until the subset of the triangulation of $P\cup Q$ that minimally covers $P\cap Q$ is fully contained in $P\cap Q$. Now, the valuation property of $\oN$ follows immediately from the definition. Thus, we have shown that  $\oN: \lp{2}\to \T^9$ is a simple, 1-homogeneous,  $\slz{2}$ equivariant, and  translation invariant valuation.
Elementary calculations show that $\oL_1^3(T_2)^3\in\T^9$ is non-trivial and not a multiple of $\oL_1^9(T_2)$.

We remark that for $r>9$ odd, we can define new valuations in a similar way using symmetric tensors products of $\oL_1^{s_j}(T_2)$ for $j=1, \dots, m$ with $s_j>1$ odd and $s_1+ \dots+s_m=r$. In general, there are linear dependencies among these new valuations.

\goodbreak
\subsection{Simple tensor valuations on $\lpn$}

Let $n\ge 3$. 
For the classification of simple tensor valuations, we use the following dissection of the 
 the 2-cylinder $T_{n-1}+[0,e_n]$  into $n$ simplices $S_1,\dots,S_n$. Let $e_0=0$. We set $S_1=T_n$ and
\begin{equation}
\label{dissectT}
S_i=[e_0+e_n,\dots,e_{i-1}+e_n,e_{i-1},\dots,e_{n-1}] \mbox{ \ for $i=2,\dots,n$}.
\end{equation}
Note that each $S_i$ is $n$-dimensional and unimodular (see, for example,  \cite[Section~2.1]{Hat02}). 

\goodbreak
Let $\oZ:\lpn\to \T^r$ be a simple, $\slnz$ equivariant, and translation invariant valuation.
Applying the dissection~(\ref{dissectT}), we make use of the translation invariance of $\oZ$ and consider $\tilde{S}_i=S_i-e_n$ for all $i>1$. Define $\phi_i\in\slnz$, for $i\in\{2,\dots,n\}$, by $\phi_i e_j=e_j$ for $j<i-1$, $\phi_i e_k=e_k-e_n$ for $i-1\leq k\leq n-1$, and $\phi_i e_n=e_{i-1}$. Let $\phi_1$ be the identity matrix. Then $\tilde{S}_i=\phi_i T_n$ for all $i\ge1$ and
\begin{equation}\label{def_diss}
\begin{array}{rl}
\!\!\!\!\oZ(T_{n-1}+[0,e_n])\!\!\!\!&=\oZ(\phi_1 T_n)(e_1[r_1],\dots, e_n[r_n])+\dots+\oZ(\phi_nT_n)(e_1[r_1],\dots, e_n[r_n])\\[6pt]
&=\oZ(T_n)(\phi_1^te_1[r_1],\dots, \phi_1^te_n[r_n])+\dots+\oZ(T_n)(\phi_n^t e_1[r_1],\dots, \phi_n^t e_n[r_n])
\end{array}
\end{equation}
for any $r_1, \dots, r_n\in\{0,\dots,r\}$ with $r_1+\dots+ r_n=r$. For $\oZ(T_{n-1}+[0,e_n])=0$, 
this is a system of linear and homogeneous equations for the $\binom{n+r-1}{r}$ coordinates of the tensor $\oZ(T_n)$. 
In addition, if $\psi\in\slnz$ is an even permutation of $e_1, \ldots, e_n$, then $\psi T_n= T_n$ and we can also make use of these symmetries.
We checked directly that the corresponding matrix has full rank and that, therefore, all coordinates vanish by using a computer algebra system (namely, SageMath~\cite{Sage}) in the following cases. 

\begin{lemma}\label{tensor_simple_nge3}
Let $\oZ:\lpn\rightarrow\T^r$ be a simple, $\slnz$ equivariant, and translation invariant valuation such that $\,\oZ(T_{n-1} +[0,e_n])=0$. If $\,3\le n<r \le 8$, then  $\oZ(T_n)=0$.
\end{lemma}

\noindent  For $n=3$, we also require the following variants of the above lemma. The calculations were again performed with a computer algebra system.

\goodbreak
\begin{lemma}\label{tensor_simple_3o}
Let $\oZ:\lp3\rightarrow\T^r$ be a simple, $\slz3\!$ equivariant, and translation invariant valuation. If $$\oZ(T_{2} +[0,e_3])(e_1[r_1], e_2[r_2], e_3[r_3])=0$$ for $r_3$ odd and $r\in\{3,5,7\}$, then $\oZ(T_3)=0$. 
\end{lemma}

\begin{lemma}\label{tensor_simple_3e}
Let $\oZ:\lp3\rightarrow\T^r$ be a simple, $\slz3\!$ equivariant, and translation invariant valuation. If $$\oZ(T_{2} +[0,e_3])(e_1[r_1], e_2[r_2], e_3[r_3])=0$$ for $r_3$ even and $r\in\{2,4,6,8\}$, then $\oZ(T_3)=0$.
\end{lemma}

\noindent
For more information, see \cite{Laura}. 

\goodbreak
The dissection (\ref{dissectT}) is also used in the proof of the following result.

\begin{lemma}\label{rlen}
Let $\oZ:\lpn\rightarrow\T^r$ be a simple, $\slnz$ equivariant, and translation invariant valuation such that $\,\oZ(T_{n-1} +[0,e_n])=0$.
If $n\ge 2$ and $n\geq r$, then $\oZ(T_n)=0$.
\end{lemma}

\begin{proof}
As in~(\ref{def_diss}), we have
\begin{align*}
0=\oZ(T_{n-1}+[0,e_n])=\oZ(\phi_1 T_n)+\dots+\oZ(\phi_nT_n),
\end{align*}
as $\oZ$ is a simple, translation invariant valuation that vanishes on $T_{n-1}+[0,e_n]$. Thus, for any coordinate of $\oZ(T_{n-1}+[0,e_n])$ where ${r_1},\dots,{r_{n-1}}\in\{0,\dots,{r}\}$, we have the equations
\begin{align}\label{dissection_n-1}
0&=\oZ(T_{n-1}+[0,e_n])(e_1[r_1],\dots,e_{n-1}[r_{n-1}])\nonumber
\\[4pt]
&=\oZ(T_n)(e_1[r_1],\dots,e_{n-1}[r_{n-1}])+\oZ(\phi_2T_n)(e_1[r_1],\dots,e_{n-1}[r_{n-1}])\nonumber
\\
&\phantom{=\oZ(T_n)(e_1[r_1],\dots,e_{n-1}[r_{n-1}])+}+\dots+\oZ(\phi_nT_n)(e_1[r_1],\dots,e_{n-1}[r_{n-1}])\nonumber
\\
&=\oZ(T_n)(e_1[r_1],\dots,e_{n-1}[r_{n-1}])+\oZ(T_n)(\phi_2^te_1[r_1],\dots,\phi_2^te_{n-1}[r_{n-1}])\\
&\phantom{=\oZ(T_n)(e_1[r_1],\dots,e_{n-1}[r_{n-1}])+}+\dots+\oZ(T_n)(\phi_n^te_1[r_1],\dots,\phi_n^te_{n-1}[r_{n-1}])\nonumber
\\
&=\oZ(T_n)(e_1[r_1],\dots,e_{n-1}[r_{n-1}])+\oZ(T_n)(e_1+e_n[r_1],e_2[r_2],\dots,e_{n-1}[r_{n-1}]) \nonumber
\\
&\phantom{=\oZ(T_n)(e_1[r_1],\dots,e_{n-1}[r_{n-1}])+}+\dots+\oZ(T_n)(e_1[r_1],\dots, e_{n-2}[r_{n-2}],e_{n-1}+e_n[r_{n-1}]).\nonumber
\end{align}
For $r_1,\dots,r_n\in\N$ such that $r_1+\dots+r_n=r$, the corresponding  coordinate of  $\oZ(T_n)$ is $\oZ(T_n)(e_1[r_1],\dots,e_{n}[r_{n}])$. As $T_n$ is invariant under permutations, the permutations of the $r_j$'s are irrelevant. Without loss of generality, we may then assume $r_1\geq r_2\geq\cdots\geq r_n$ and drop $r_j$ when $r_j=0$ from our notation. Set $a_{r_1,\dots,r_m}=\oZ(T_n)(e_{i_1}[r_1],\dots,e_{i_m}[r_{m}])$ where $r_1+\dots+r_m=r$ and each $r_j\ge 1$.

We define a total order $\preceq$ on the coordinates $a_{r_1,\dots,r_m}$ by saying that $a_{r_1,\dots,r_m}\preceq a_{s_1,\dots,s_{m'}}$ if $r_1< s_1$ or if $r_1=s_1, \dots, r_{j-1}= s_{j-1}$ and $r_j < s_j$. Therefore, the coordinates are ordered in the following way from the biggest to smallest:
\[
a_{r},a_{r-1,1},a_{r-2,2},\dots,a_{\lceil \frac{r}{2} \rceil,\lfloor \frac{r}{2} \rfloor},a_{r-2,1,1},\dots,a_{\lceil \frac{r-1}{2} \rceil,\lfloor \frac{r-1}{2} \rfloor,1},\dots,a_{2,1,\dots,1},a_{1,\dots,1},
\]
where $\lfloor x\rfloor$ is the largest integer less or equal to $x$ and $\lceil x \rceil$ is the smallest integer greater or equal to $x$.

\goodbreak
We claim that the equations (\ref{dissection_n-1})  imply that the coordinates of $\oZ(T_n)$ that involve at most $(n-1)$ of $e_1, \dots, e_n$ all vanish. 
One can see this by noticing that, for given $r_1, \dots, r_m$ with $m<n$, the linear equation (\ref{dissection_n-1})
 only involves $a_{r_1,\dots,r_m}$ and coordinates that are smaller than this coordinate in the ordering defined above. Thus, for $r<n$, we have an equation for each coordinate and the system of equations can be regarded as an upper-triangular matrix that, therefore, has full-rank. Thus, each coordinate vanishes implying that $\oZ(T_n)=0$ for $r<n$.

Additionally, for $n=r$, we have
\begin{eqnarray*}
0&=&\oZ({T}_{n-1}+[0,e_n])(e_1,e_2,\dots,e_n)\\
&=&\oZ(T_n)(e_1,\dots,e_n)+\oZ(T_n)(e_1+e_n,e_2,\dots,e_{n-1},-(e_1+\dots+e_{n-1}))\\
&&+\oZ(T_n)(e_1,e_2+e_n,\dots,e_{n-1},-(e_2+\dots+e_{n-1}))\\
&&+\oZ(T_n)(e_1,e_2,e_3+e_n,\dots,e_{n-1},-(e_3+\dots+e_{n-1})) +\cdots\\
&&+\oZ(T_n)(e_1,\dots,e_{n-2},e_{n-1}+e_n,-e_{n-1})\\
&=&-(n-2) a_{1,\dots,1}-(n-1)^2 a_{2,1,\dots,1}.
\end{eqnarray*}
Thus $a_{1,\dots,1}=0$, as $a_{2,1,\dots,1}=0$ by the first step. As the first step also shows that all further coordinates vanish, this completes the proof.
\end{proof}

We now establish the three-dimensional case first and then, using this result, the general case.

\begin{lemma}\label{tensor_three}
Let $\,\oZ:\lp{3}\rightarrow\T^r$ be a simple, $\slz3$ equivariant, and translation invariant valuation. If $\,2\le r \le 8$, then $\oZ(T_3)=0$.
\end{lemma}

\begin{proof}
We only need to consider the statement for $\oZ$ being in addition $i$-homogeneous by Theorem~\ref{Polynomial}. If $\oZ$ is $3$-homogeneous, then it is trivial due to Proposition~\ref{trans_van}.  Lemma \ref{zon} implies that $\oZ(T_2+[0,e_3])=0$ if $i=1$.  Hence, Lemma \ref{tensor_simple_nge3} and Lemma \ref{rlen} imply that $\oZ(T_3)=0$ for $i=1$. Therefore, let $\oZ$ be $2$-homogeneous. 

Since  $\oZ$ is simple and $2$-homogeneous, Theorem \ref{McM_rec} implies that $\oZ(Q)= -\oZ(-Q)$, that is, $\oZ$ is \emph{odd}.
Using that $\oZ$ is odd, translation invariant, and $\slz3$ equivariant, we obtain
\begin{eqnarray}\label{odd}
\oZ(T_2+[0,e_3])(e_1[r_1], e_2[r_2], e_3[r_3])&=&-\oZ(-T_2-[0,e_3])(e_1[r_1], e_2[r_2], e_3[r_3])\nonumber\\
&=&-\oZ(-T_2+[0,e_3])(e_1[r_1], e_2[r_2], e_3[r_3])\\
 &=&(-1)^{r_1+r_2+1}\oZ(T_2+[0,e_3])(e_1[r_1], e_2[r_2], e_3[r_3]).\nonumber
\end{eqnarray}

First, let $r$ be odd. Then (\ref{odd}) implies that for $r_3$ odd,
\begin{equation*}
\oZ(T_2+[0,e_3])(e_1[r_1], e_2[r_2], e_3[r_3])=0.
\end{equation*}
We can therefore apply Lemma \ref{tensor_simple_3o} and obtain that $\oZ(T_3)=0$. This  implies the statement of the lemma for $r$ odd.

Second, let $r$ be even. Then (\ref{odd}) implies that for $r_3$ even,
\begin{equation*}
\oZ(T_2+[0,e_3])(e_1[r_1], e_2[r_2], e_3[r_3])=0.
\end{equation*}
Applying Lemma \ref{tensor_simple_3e} gives $\oZ(T_3)=0$. This completes the proof of the lemma.
\end{proof}

\goodbreak
\begin{proposition}\label{tensor_simple}
 Let $\,\oZ:\lpn\rightarrow\T^r$ be a simple, $\slnz$ equivariant, and translation invariant valuation. If $n\ge 3$ and $\,2\le r\le 8$, then $\oZ=0$.
\end{proposition}

\begin{proof}
For $n=3$, we have $\oZ(T_3)=0$ by Lemma \ref{tensor_three} and the result follows from Corollary~\ref{BetkeKneserinvariance}. 

Let $n>3$ and suppose that the statement is true in dimension $n-1$. Let $\,\oZ:\lpn\rightarrow\T^r$ be a simple, $\slnz$ equivariant, and translation invariant valuation for $2\le r\le 8$. 
We only need to consider the statement for $\oZ$ being in addition $i$-homogeneous by Theorem~\ref{Polynomial}. 
If $\oZ$ is $n$-homogeneous, then it is trivial due to Proposition~\ref{trans_van}. So, let $1\le i\le n-1$.

Define $\oY: \lp{n-1} \to \T^s(\R^{n-1})$ by setting for $P\in\lp{n-1}$ 
\begin{equation*}\label{cyl_val}
\oY(P)(e_1[r_1], \dots, e_{n-1}[r_{n-1}])=\oZ(P+ [0,e_n])(e_1[r_1], \dots, e_{n-1}[r_{n-1}], e_n[r_n])
\end{equation*}
where $r_1+\dots+r_{n-1}=s$ and $r_1+\dots+r_n=r$.
Then $\oY$ is a simple, $\slz{n-1}$ equivariant, and  translation invariant valuation. 
Furthermore, $\oY$ is $(i-1)$-homogeneous as 
\[
k^i\oY(P)=k^i\oZ(P+[0,e_n])=\oZ(k\,P+k\,[0,e_n]))=k\oZ(kP+[0,e_n])=k\oY(kP)
\]
by the simplicity and translation invariance of $\oZ$.

For $2\le s\le 8$, the induction assumption implies that $\oY=0$. 
If $s=0$, then $\oY$ is real-valued and  $\slz{n-1}$ and translation invariant. Since it is simple, the Betke \& Kneser Theorem implies that it is a multiple of the $(n-1)$-dimensional volume as, by Lemma \ref{coefficients}, the only simple Ehrhart coefficient is volume. Hence, $\oY$ is also $(n-1)$-homogeneous and must vanish.
If $s=1$, then $\oY$ is vector-valued and $\slz{n-1}$ equivariant and translation invariant. Since it is simple, Theorem \ref{vector} implies that it is a multiple of the moment vector, as by Lemma \ref{coefficients} the only simple  Ehrhart tensor of rank one is the moment tensor. Thus $\oY$ is also $n$-homogeneous, which implies that it vanishes.

In particular, we obtain $\oZ(T_{n-1}+[0,e_n])=0$. Hence, Lemma \ref{tensor_simple_nge3} and Lemma \ref{rlen} imply that $\oZ(T_n)=0$. The result now follows from Corollary \ref{BetkeKneserinvariance}. 
\end{proof}

\subsection{General tensor  valuations}

Let $r\ge 2$. The translation property (\ref{L_q^r}) combined with (\ref{L0}) gives
\[
\oL_1^r(P+x)=\oL_1^r(P)+\oL_0^{r-1}(P)x=\oL_1^r(P),
\]
that is, $\oL_1^r$ is translation invariant.
We show that every $\slnz$ equivariant and translation invariant valuation is a multiple of $\oL_1^r$ for $2\le r\le 8$. We start with the case of 1-homogeneous valuations.

\begin{proposition}\label{1_hom}
Let $\oZ:\lpn\rightarrow\T^r$ be an $\slnz$ equivariant and  translation invariant valuation.
If $\,\oZ$ is $1$-homogeneous and $2\le r\le 8$, then there exists $c\in\R$ such that $\oZ=c\oL_1^r$.
\end{proposition}

\begin{proof}
We use induction on the dimension $n$. The case $n=1$ is elementary (and also follows from the Betke \& Kneser Theorem) and states that, for a 1-homogeneous and translation invariant valuation $\oZ: \lp{1} \to \T^r(\R)$, we have  $\oZ=c\oL_1^r$ for some $c\in\R$. 

\goodbreak
Assume the statement holds for $n-1$. Restrict $\oZ$ to lattice polytopes with vertices in $\Z^{n-1}$. By  Lemma~\ref{high_coord}, we may view this restricted valuation as a function $\oZpr: \lp{n-1} \to \T^r(\R^{n-1})$. Since  $\oZpr$  is an $\operatorname{SL}_{n-1}(\Z)$ equivariant and translation invariant valuation on $\lp{n-1}$, by the induction hypothesis, there is $c\in\R$ such that $\oZpr(P)=c\oL_1^r(P)$ for  $P\in\lp{n-1}$.  By Lemma~\ref{high_coord},  for $1\le r_n\le r$, 
$$\oZ(P)(e_{i_1},\dots, e_{i_{r-r_n}},e_n[r_n])=c \oL_1^r(P)(e_{i_1},\dots, e_{i_{r-r_n}},e_n[r_n])=0$$ 
where $P\in\lp{n-1}$ and  $i_1, \dots, i_{r-r_n} \in\{1, \dots, n-1\}$. Hence $\oZ(P)=c\oL_1^r(P)$ for $P\in\lp{n-1}$. 
\goodbreak\noindent
Set
\[
\widetilde{\oZ}=\oZ-c\oL_1^r.
\]
Note that $\widetilde{\oZ}$ vanishes on $\lp{n-1}$. By the $\slnz$ equivariance and translation invariance of  $\widetilde{\oZ}$, this implies that $\widetilde{\oZ}$ vanishes on lattice polytopes in any $(n-1)$-dimensional lattice hyperplane $H\subset \R^n$ as we have $H\cap \Z^n=\phi \,\Z^{n-1}+x$ for some $\phi\in\slnz$ and $x\in\Z^n$. In other words, $\widetilde{\oZ}$ is simple and the statement follows from Propositions \ref{tensor_simple_planar} and \ref{tensor_simple}.
\end{proof}

\goodbreak
We can now extend Proposition~\ref{trans_van} to $i$-homogeneous valuations with $2\leq i\leq n$. 

\begin{proposition}\label{all_hom}
Let $\oZ:\lpn\rightarrow\T^r$ be an $\slnz$ equivariant and  translation invariant valuation.
If $\,\oZ$ is $i$-homogeneous with $2\leq i\leq n$ and $2\le r\le 8$, then $\oZ=0$.
\end{proposition}

\begin{proof}
Lemma~\ref{tinv_simple}  implies that the valuation $\oZ$ vanishes on all lattice polytopes of dimension $m<i$. We use induction on $m$ and show that $\oZ$ also vanishes on all $m$-dimensional lattice polytopes for $i\le m\le n$.

\goodbreak
First, let $m=i$. Restrict $\oZ$ to lattice polytopes in $\lp{i}$. By  Lemma~\ref{high_coord}, we may view this restricted valuation as a function $\oZpr: \lp{i} \to \T^r(\R^{i})$. Since $\oZpr$ is invariant under translations of $\lp{i}$ into itself and  $\slz{i}$ equivariant,  Proposition~\ref{trans_van} implies that $\oZpr$ vanishes on $\lp{i}$. Thus, by Lemma~\ref{high_coord}, we  obtain that also $\oZ$ vanishes on lattice polytopes with vertices in $\Z^i$.   Now, let $P\in\lpn$ be a general  $i$-dimensional lattice polytope.  Let $H$ be the $i$-dimensional subspace of $\R^n$ that is parallel to the affine hull of $P$. There exists $x\in\Z^n$ such that $P+x\in H$ and there exists  $\phi\in\slnz$ such that $\phi(H\cap\Z^n)=\Z^i$. Since $\oZ$ vanishes on  $\lp{i}$ and is $\slnz$ equivariant and translation invariant, we obtain that $\oZ(P)=0$.

\goodbreak
Next, for  $m>i$, suppose that $\oZ(Q)=0$ for all $Q\in\lpn$ with $\dim(Q)<m$. By  Lemma~\ref{high_coord}, we may view the restriction of $\oZ$ to lattice polytopes in $\lp{m}$ as a function $\oZpr: \lp{m} \to \T^r(\R^{m})$. Since $\oZpr$ is a simple,  $\slz{m}$ equivariant and translation invariant valuation and  $m\geq 3$,  Proposition~\ref{tensor_simple} implies that $\oZpr$ vanishes on  $\lp{m}$. As in the previous step, this implies that  $\oZ(P)=0$ for any $m$-dimensional lattice polytope in $\lpn$ and, by induction, we have $\oZ=0$.
\end{proof}

\goodbreak

The characterization of $\oL_1^r$ follows immediately from the combination of Theorem~\ref{Polynomial}, Proposition~\ref{1_hom}, and Proposition~\ref{all_hom}. 

\begin{corollary}\label{trans_inv}
For $2\le r \le 8$, a function $\oZ:$~$\lpn\rightarrow$~$\T^r$ is an $\,\slnz$ equivariant and translation invariant valuation if and only if there exists $c\in\R$ such that $\oZ=c\oL_1^r$.
\end{corollary}

Together with Proposition~\ref{linear}, we obtain the following consequence of Corollary~\ref{trans_inv}.

\begin{corollary}\label{mink_add}
For $2\le r \le 8$, a function $\oZ:\lpn\to\T^r$ is  $\slnz$ equivariant, translation invariant, and Minkowski additive if and only if there exists $c\in\R$ such that $\oZ=c\oL_1^r$.
\end{corollary}

\subsection{Proof of Theorem~\ref{tensor}}

The classification is now obtained by an inductive proof on the rank $r$. Recall that the Betke \& Kneser Theorem gives the characterization for the case $r=0$ and Theorem \ref{vector} for $r=1$. The induction assumption gives that 
$$\oZ^{r-1}=\sum_{i=1}^{n+r-1}c_{i}\oL_i^{ r-1}$$ 
for some constants $c_{1},\dots,c_{n+r-1}\in\R$. Furthermore, for any $y\in\Z^n$, this characterization together with Proposition \ref{assoc1} applied to $\oZ^{r-1}$ and to $\oL_i^{r-1}$  yields
\begin{eqnarray*}
\oZ^{r-1}(P+y)&=&\oZ^{r-1}(P)+\oZ^{r-2}(P)\frac{y}{1!}+\cdots+\oZ^0(P)\frac{y^{r-1}}{(r-1)!}\\
&=&\sum_{i=1}^{n+r-1}c_{i}\oL_i^{ r-1}(P+y)\\
&=&\sum_{i=1}^{n+r-1}c_{i} \,\Big(\oL_i^{ r-1}(P)+\oL_{i-1}^{ r-2}(P)\frac{y}{1!}+\cdots+\oL_{i-r+1}^{ 0}(P)\frac{y^{r-1}}{(r-1)!}\Big).
\end{eqnarray*}
A comparison of the coefficients of the polynomial expansion in $y$ gives 
\[
\oZ^{r-j}(P)=\sum_{i=1}^{n+r-1}c_{i}\oL_{i-j+1}^{r-j}(P).
\]
Consider the $\slnz$ equivariant valuation
\[
\widetilde{\oZ}=\oZ-\sum_{i=2}^{n+r}c_{i-1}\oL_{i}^{ r}.
\]
For $y\in\Z^n$, by Proposition \ref{assoc1} and the induction assumption, we obtain
\begin{eqnarray*}
\widetilde{\oZ}(P+y)&=&\oZ(P+y)-\sum_{i=2}^{n+r}c_{i-1}\oL_{i}^{ r}(P+y)\\[-2pt]
&=&\oZ(P)+\sum_{j=1}^r\oZ^{r-j}(P)\frac{y^j}{j!}-\sum_{i=2}^{n+r}\sum_{j=0}^rc_{i-1}\oL_{i-j}^{ r-j}(P)\frac{y^j}{j!}\\
&=&\oZ(P)+\sum_{j=1}^r\sum_{i=1}^{n+r-1}c_{i}\oL_{i-j+1}^{ r-j}(P)\frac{y^j}{j!}-\sum_{i=2}^{n+r}c_{i-1}\oL_{i}^{ r}(P)-\sum_{i=2}^{n+r}\sum_{j=1}^rc_{i-1}\oL_{i-j}^{ r-j}(P)\frac{y^j}{j!}\\
&=&\oZ(P)-\sum_{i=2}^{n+r}c_{i-1}\oL_{i}^{ r}(P)\\[2pt]
&=&\widetilde{\oZ}(P).
\end{eqnarray*}
Consequently, the function $\widetilde{\oZ}$ is translation invariant and Corollary~\ref{trans_inv} implies that $\widetilde{\oZ}=c_1\oL^r_1$ proving the theorem.

\goodbreak

\section*{Acknowledgments} The work of both authors was supported, in part, by the Austrian Science Fund (FWF) Projects P25515-N25 and I3017-N35.

\footnotesize

\bigskip
\bigskip
\parindent=0pt
\parbox[t]{8.5cm}{
\begin{samepage}
Monika Ludwig\\
Institut f\"ur Diskrete Mathematik und Geometrie\\
Technische Universit\"at Wien\\
Wiedner Hauptstra\ss e 8-10/1046\\
1040 Wien, Austria\\
E-mail: monika.ludwig@tuwien.ac.at
\end{samepage}}
\parbox[t]{7cm}{
\begin{samepage}
Laura Silverstein\\
Institut f\"ur Diskrete Mathematik und Geometrie\\
Technische Universit\"at Wien\\
Wiedner Hauptstra\ss e 8-10/1046\\
1040 Wien, Austria\\
E-mail: laura.silverstein@tuwien.ac.at
\end{samepage}}

\end{document}